\documentclass[11pt,a4paper]{amsart}
\usepackage[T1]{fontenc}
\usepackage{lmodern}
\usepackage[margin=20mm]{geometry}
\usepackage{amsmath,amssymb,amscd,mathtools}
\usepackage{microtype}
\usepackage{enumitem}
\usepackage{xcolor}
\usepackage[colorlinks=true,linkcolor=blue!55!black,
  citecolor=blue!55!black,urlcolor=blue!55!black]{hyperref}
\setlist[enumerate]{label=\textup{(\arabic*)},leftmargin=2.3em}

\newtheorem{theorem}{Theorem}[section]
\newtheorem{lemma}[theorem]{Lemma}
\newtheorem{proposition}[theorem]{Proposition}
\newtheorem{corollary}[theorem]{Corollary}
\newtheorem{introtheorem}{Theorem}

\theoremstyle{definition}
\newtheorem{definition}[theorem]{Definition}
\newtheorem{conjecture}[theorem]{Conjecture}
\theoremstyle{remark}
\newtheorem{remark}[theorem]{Remark}
\numberwithin{equation}{section}

\DeclareMathOperator{\Spa}{Spa}
\DeclareMathOperator{\Spec}{Spec}
\DeclareMathOperator{\Lie}{Lie}
\DeclareMathOperator{\End}{End}

\DeclareMathOperator{\GL}{GL}
\DeclareMathOperator{\Gal}{Gal}
\DeclareMathOperator{\Hom}{Hom}
\DeclareMathOperator{\im}{im}
\DeclareMathOperator{\Tr}{Tr}
\DeclareMathOperator{\Vect}{Vect}
\DeclareMathOperator{\Rig}{Rig}

\newcommand{\Cp}{\mathbb C_p}
\newcommand{\Qp}{\mathbb Q_p}
\newcommand{\Zp}{\mathbb Z_p}
\newcommand{\Z}{\mathbb Z}
\newcommand{\OO}{\mathcal O}
\newcommand{\Oh}{\widehat{\mathcal O}}
\newcommand{\TT}{\mathbb T}
\newcommand{\proet}{\text{pro-\'et}}
\newcommand{\la}{G\text{-}\mathrm{la}}
\newcommand{\jla}{(G\times\Gamma)\text{-}\mathrm{la}}

\title{On the Surjectivity of the Geometric Sen Map of Perfectoid Torsors}
\author{Tian Qiu}
\address{Beijing International Center for Mathematical Research, Peking University, Beijing, China}
\email{\nolinkurl{qiu.tian@pku.edu.cn}}
\author{Jiahong Yu}
\address{Academy of Mathematics and Systems Science, Chinese Academy of
Sciences, Beijing, China; Morningside Center of Mathematics, Chinese Academy
of Sciences, Beijing, China}
\email{\nolinkurl{yu_jh@amss.ac.cn}}
\date{}

\begin{document}

\begin{abstract}
We study surjectivity of the geometric Sen map for perfectoid
$G$-torsors over smooth rigid analytic spaces over $\Cp$, where
$G$ is a compact $p$-adic Lie group. We construct an affinoid
perfectoid $\Zp^2$-torsor over a smooth curve whose Sen map is
not surjective, thereby disproving Rodr\'\i guez Camargo's
conjectured implication from perfectoidness to surjectivity.
We prove that the Sen map of every perfectoid $G$-torsor over
a smooth curve is nevertheless surjective on a Zariski open
dense subset. Finally, we show that locally analytic
decompletion implies Sen surjectivity for diamondian affinoid
perfectoid torsors.
\end{abstract}

\maketitle
\setcounter{tocdepth}{1}
\tableofcontents

\section{Introduction}

\subsection{Arithmetic motivation}\label{sec:intro-arithmetic}
Let $K/\Qp$ be a finite extension and let $L/K$ be a Galois extension with $p$-adic Lie Galois
group $\Gamma$. Put $\mathfrak g=\Lie\Gamma$.
Decompletion along the cyclotomic tower yields the arithmetic Sen
element $\Theta_{L/K}\in\Cp\otimes_{\Qp}\mathfrak g$
\cite{Sen1973,Sen1980}.
We view this element as a $\Cp$-linear functional on the dual Lie
algebra:
\[
 \theta_{L/K}:\Cp\otimes_{\Qp}\mathfrak g^\vee
 \longrightarrow\Cp,
 \qquad \lambda\longmapsto\lambda(\Theta_{L/K}).
\]
Thus $\theta_{L/K}$ is surjective if and only if
$\Theta_{L/K}\ne0$.
A fundamental theorem in the $p$-adic Hodge theory is the relationship between the perfectoidness of $\widehat L$ and $\theta_{L/K}$. Precisely, write $G_K$ for the absolute Galois group of $K$, $I_K\subseteq G_K$
for its inertia subgroup, and $I(L/K)$ for the image of $I_K$ in
$\Gamma$. For a finite-dimensional continuous representation
$\rho_V:G_K\to\GL(V)$ over $\Qp$, Sen's criterion states that
\begin{equation}\label{eq:intro-sen-zero}
 \Theta_V=0
 \quad\Longleftrightarrow\quad
 V\text{ is }\Cp\text{-admissible}
 \quad\Longleftrightarrow\quad
 \rho_V(I_K)\text{ is finite},
\end{equation}
where $\Theta_V\in \mathrm{End}_{\mathbb C_p}(V\otimes_{\mathbb Q_p}\Cp)$ is the Sen operator of $V$; see \cite[\S~5]{Sen1973} and
\cite[Remark~1.4(ii)]{BergerColmez2016}.
Here the middle condition means that
$\dim_K(\Cp\otimes_{\Qp}V)^{G_K}=\dim_{\Qp}V$.
For a representation $\rho:\Gamma\to\GL(V)$, the Sen operator of
the composite $G_K\to\Gamma\xrightarrow{\rho}\GL(V)$ is
$d\rho(\Theta_{L/K})$.
Sen's ramification theorem \cite{Sen1972}, together with
\eqref{eq:intro-sen-zero}, therefore gives
\begin{equation}\label{eq:intro-arithmetic-perfectoid}
 \widehat L\text{ is perfectoid}
 \quad\Longleftrightarrow\quad
 I(L/K)\text{ is infinite}
 \quad\Longleftrightarrow\quad
 \theta_{L/K}\text{ is surjective};
\end{equation}
see \cite[Proposition~6.13 and Remark~6.14]{He2026}.

Another important property of p-adic Lie Galois extension is the locally analytic decompletion:  Berger and Colmez prove that,
for every finite-dimensional continuous semilinear
$\widehat L$-representation $W$ of $\Gamma$, the multiplication map
\begin{equation}\label{eq:intro-arithmetic-decompletion}
 \widehat L\otimes_{\widehat L^{\Gamma\text{-}\mathrm{la}}}
 W^{\Gamma\text{-}\mathrm{la}}
 \xrightarrow{\ \sim\ }W
\end{equation}
is an isomorphism \cite[Theorem~1.7]{BergerColmez2016}. Here
$\widehat L^{\Gamma\text{-}\mathrm{la}}$ is a field and
$W^{\Gamma\text{-}\mathrm{la}}$ is finite-dimensional over it.
Combining these two theorems, perfectoidness of $\widehat L$ implies both surjectivity of
the arithmetic Sen map and locally analytic decompletion.

These arithmetic results lead to the geometric questions studied
in this paper: does perfectoidness force surjectivity of the Sen
map, and how is this surjectivity related to locally analytic
decompletion?
Pan developed a geometric Sen operator for curves \cite{Pan2022},
and Rodr\'\i guez Camargo constructed geometric Sen operators for
general log-smooth rigid analytic spaces \cite{RodriguezCamargo2026}.
In the smooth setting, let $X/\Cp$ be a smooth rigid analytic space and
let $\pi:\widetilde X\to X$ be a profinite \'etale torsor under a
compact $p$-adic Lie group $G$. Put $\mathfrak g=\Lie G$ and let
$\mathfrak g_\pi^\vee$ be the coadjoint local system. The geometric
Sen map is the morphism on $X_{\proet}$
\[
 \theta_{\widetilde X}:
 \Oh_X\otimes_{\Qp}\mathfrak g_\pi^\vee
 \longrightarrow
 \Oh_X\otimes_{\OO_X}\Omega^1_{X/\Cp}(-1).
\]
The arithmetic map $\theta_{L/K}$ takes values in the
one-dimensional space $\Cp$; its geometric analogue takes values
in the cotangent directions of $X$. The perfectoidness criterion
\eqref{eq:intro-arithmetic-perfectoid} therefore motivates the
following conjecture.

\begin{conjecture}[Rodr\'\i guez Camargo]
\label{conj:camargo-perfectoid}
Let $X/\Cp$ be a smooth rigid analytic space, and let
$\widetilde X\to X$ be a profinite \'etale torsor under a compact
$p$-adic Lie group. Then the geometric Sen map is surjective if and
only if $\widetilde X$ is perfectoid.
\end{conjecture}

This is the smooth case of
\cite[Conjecture~3.3.5]{RodriguezCamargo2026}.
The precise tower notions used here are given in
Definition~\ref{def:four-perfectoid-notions}.

\subsection{A perfectoid torsor with nonsurjective Sen map}
\label{sec:intro-counterexample}
The arithmetic analogy and the decompletion criterion make
Conjecture~\ref{conj:camargo-perfectoid} a natural expectation.
Tongmu He proved an implication from Sen surjectivity to
Riemann--Zariski stalkwise pre-perfectoidness \cite{He2026}, and
Bellovin, Cai and Howe established the conjectural
characterization for covers of abelian varieties and related
commutative rigid analytic groups \cite{BCH25}. In forthcoming joint work of the first author with Chen, Li, and Yang, it is proved that if $X$ is smooth of dimension $n$ and $G=\mathbb{Z}_p^n$, then the perfectoidness of $\widetilde X$ implies the surjectivity of $\theta_{\widetilde X}$.
Nevertheless, the implication from perfectoidness to surjectivity
fails in general, even in dimension one.

\begin{introtheorem}[ = Theorem \ref{thm: perfd not surjective} ]\label{thm:counterexample}
There exist a smooth connected rigid analytic curve $X/\Cp$ and
an affinoid perfectoid profinite \'etale $\Zp^2$-torsor
$\widetilde X\to X$ whose geometric Sen map is not surjective.
\end{introtheorem}

Here affinoid perfectoid is meant in the sense of
\cite[Definition~4.3(i)]{Scholze2013}. The base of the construction is the closed unit disc with parameter $t$. We pull a
two-dimensional toric tower back along the cusp parametrization
$t\mapsto(t^2,t^3)$. This morphism is finite and universally
injective, so the pullback remains affinoid perfectoid by
Theorem~\ref{thm:ui-pullback}. Its differential has image
$t\OO_X\,dt$. Functoriality forces the Sen map to have image in
$t\Oh_X\,dt(-1)$, and it is therefore not surjective at the origin.
Theorem~\ref{thm:intro-decompletion} consequently also gives a
failure of locally analytic decompletion near that point.

Our example satisfies the strong affinoid perfectoid condition.
This distinguishes it from the example in
\cite[Remark~5.5]{BCH25}, which is diamondian
affinoid perfectoid but not affinoid perfectoid; see
Remark~\ref{rem:bchh-example}.

\subsection{Generic surjectivity for curves}\label{sec:intro-generic}
In the counterexample, the Sen map is surjective away from the
origin. This suggests allowing a proper closed analytic exceptional
locus in the perfectoid-to-surjective direction of
Conjecture~\ref{conj:camargo-perfectoid}.

\begin{conjecture}[Generic Sen surjectivity]\label{conj:generic-surjectivity}
Let $X/\Cp$ be a smooth connected rigid analytic space, and let
$\widetilde X\to X$ be a profinite \'etale torsor under a compact
$p$-adic Lie group. Assume that its tower is perfectoid in the sense
of Definition~\ref{def:four-perfectoid-notions}\textup{(3)}.
Then there exists a Zariski open dense subset $U\subseteq X$ such
that $\theta_{\widetilde X}|_{U_{\proet}}$ is surjective.
\end{conjecture}

We prove this conjecture for smooth curves.

\begin{introtheorem}[ = Theorem \ref{thm:curve} ]\label{thm:intro-curve}
Let $X/\Cp$ be a smooth connected rigid analytic curve, and let
$\widetilde X\to X$ be a perfectoid profinite \'etale torsor under
a compact $p$-adic Lie group. Then its geometric Sen map is surjective
at every type~2 point. Its nonsurjectivity locus is a locally finite
closed analytic subset consisting of classical points, and is finite
if $X$ is quasicompact. In particular,
Conjecture~\ref{conj:generic-surjectivity} holds for $X$.
\end{introtheorem}

A failure at a type~2 point would give a representation with
injective differential and a period matrix over the perfectoid
completion of the corresponding branch extension. We transfer
Wewers's ramification estimates to type~2 fields and completed
unramified Lie extensions, and adapt Sen's ramification argument
to show that the inertia image of such a representation is finite.
This contradicts perfectoidness through a residue-field obstruction.
The analytic rank locus then yields the asserted exceptional subset.

\subsection{The Hodge--Tate stack and locally analytic decompletion}
\label{sec:intro-decompletion}
For a perfectoid profinite \'etale $G$-torsor $\widetilde X\to X$,
Rodr\'\i guez Camargo's announced comparison gives a canonical map
\[
 \alpha_\pi:X^{\mathrm{HT}}\longrightarrow
 [\widetilde X^{\mathrm{an}}/G^{\mathrm{an}}]
\]
which is an isomorphism precisely when $\theta_{\widetilde X}$ is
surjective \cite[approximately 1:15:00]{RodriguezCamargoGeometricSenLecture}.
Here $X^{\mathrm{HT}}$ is the analytic Hodge--Tate stack, and
$\widetilde X^{\mathrm{an}}$ is the locally analytic model of the
tower. Thus the geometric Sen condition is equivalent to
\begin{equation}\label{eq:intro-ht-quotient}
 X^{\mathrm{HT}}\simeq
 [\widetilde X^{\mathrm{an}}/G^{\mathrm{an}}].
\end{equation}
The comparison between vector bundles on $X^{\mathrm{HT}}$ and
finite locally free $\Oh_X$-modules, together with the affinoid
realization of \eqref{eq:intro-ht-quotient}, then implies locally
analytic decompletion; the comparison inputs and their sources are
recalled in Subsection~\ref{sec:forward}.
Explicitly, if $X$ is small affinoid, $\widetilde X$ is affinoid
perfectoid, and $B=\Oh_X(\widetilde X)$, surjectivity implies that
every finite projective continuous semilinear $B$-representation $W$
of $G$ satisfies
\begin{equation}\label{eq:intro-geometric-decompletion}
 W^{\la}\text{ is finite projective over }B^{\la},
 \qquad B\otimes_{B^{\la}}W^{\la}\xrightarrow{\ \sim\ }W.
\end{equation}

The implication from Sen surjectivity to decompletion is therefore
provided by Rodr\'\i guez Camargo's comparison. Our first result
is the converse implication.

\begin{introtheorem}[ =  Theorem~\ref{thm:decompletion-surjectivity}]
\label{thm:intro-decompletion}
Let $X/\Cp$ be a small smooth affinoid, let $G$ be a compact
$p$-adic Lie group, and let $\widetilde X\to X$ be a diamondian
affinoid perfectoid profinite \'etale $G$-torsor in the sense of
Definition~\ref{def:four-perfectoid-notions}\textup{(2)}. Put
$B=\Oh_X(\widetilde X)$. If locally analytic decompletion
\eqref{eq:intro-geometric-decompletion} holds for every finite
projective continuous semilinear $B$-representation $W$ of $G$,
then $\theta_{\widetilde X}$ is surjective.
\end{introtheorem}

For affinoid perfectoid torsors, Theorem~\ref{thm:intro-decompletion}
and the forward comparison give the equivalence recorded in
Corollary~\ref{thm:equivalence}.
Our proof uses a unipotent local system obtained from a toric chart.
The Sen annihilation identity places the coefficients of its locally
analytic sections in the image of the torsor's Sen map. If these
sections generate the test module, that image contains every
cotangent direction.

\begin{remark}
The announced comparison results discussed above provide
the implication from Sen surjectivity to locally analytic
decompletion. The proofs of our main theorems  are independent of these comparisons.
\end{remark}

\medskip
\noindent\textbf{Notation and conventions.}
Unless another ground field is explicitly specified, rigid analytic
spaces are separated adic spaces over $\Cp$, locally of the form
$\Spa(A,A^\circ)$ with $A$ a rigid affinoid algebra. We write
$\mathcal M(A)$ for the Berkovich spectrum, $A^\circ$ for the
power-bounded elements, and
\[
 \rho_A(a)=\lim_{n\to\infty}\lVert a^n\rVert^{1/n}
\]
for the spectral seminorm. We write $R^{\mathrm{sn}}$ for the
seminormalization of a ring $R$ \cite[\S~1.4]{KedlayaLiu2019}.
Points mean adic points unless otherwise stated. Classical points
are also called rigid points; their set in $T$ is $T_{\mathrm{rig}}$.
For a rank-one point $x$, $\mathcal H(x)$ is its completed residue
field. A type~2 field over $\Cp$ means the completed residue field
of a type~2 point on a smooth $\Cp$-curve. For a valued field $F$,
$\kappa(F)$ denotes its residue field.

Locally analytic vectors are taken over $\Qp$. Finite projective
modules over Banach algebras carry their natural Banach topology,
and all semilinear actions are continuous. Tensor products between
finite projective modules are algebraic unless explicitly completed.
Structure and differential sheaves on $X$ also denote their pullbacks
to $X_{\proet}$. For a torsor $\pi$, the notation
$\mathfrak g_\pi$ (respectively $\mathfrak g_\pi^\vee$) denotes its
adjoint (respectively coadjoint) local system.

A smooth affinoid $X=\Spa(A,A^\circ)$ of dimension $d$ is called
\emph{small} if it admits a chart
\[
 X\longrightarrow\TT^d_{\Cp}
 =\Spa\bigl(\Cp\langle T_1^{\pm1},\ldots,T_d^{\pm1}\rangle,
            \Cp^\circ\langle T_1^{\pm1},\ldots,T_d^{\pm1}\rangle\bigr)
\]
which is a composition of rational localizations and finite \'etale
morphisms. We fix a compatible system of $p$-power roots of unity,
identifying the group of a toric tower with $\Gamma=\Zp^d$.
For a ringed site or an analytic stack $Z$, write
$\mathrm D_{\mathrm{perf}}(Z)$ for its category of perfect complexes and
$\Vect(Z)$ for its category of finite locally free modules.

\addtocontents{toc}{\protect\setcounter{tocdepth}{-1}}
\section*{Acknowledgements and statement on the use of AI}

The counterexample in Theorem~\ref{thm:counterexample} was
discovered by Rethlas running GPT-5.6 Sol.
In the proof of Theorem~\ref{thm:intro-curve}, the key step,
Lemma~\ref{lem:unramified-approximation}, was discovered by the
authors. For the subsequent analogue of Sen's filtration theorem,
the authors decomposed the argument into lemmas, and GPT-6 Astra
completed their proofs. For Theorem~\ref{thm:intro-decompletion}, the authors supplied
GPT-6 Astra with an analogous result for geometric valuation rings
from their ongoing joint work. GPT-6 Astra then independently
produced the proof of Theorem~\ref{thm:intro-decompletion}.

The authors would like to thank the Rethlas team, namely Haocheng Ju, Jiedong
Jiang, Shurui Liu, Guoxiong Gao, Yuefeng Wang, Zeming Sun, Leheng Chen, Bin Wu,
Liang Xiao, and Bin Dong, for their contributions to the development of
Rethlas \cite{JuEtAl2026Rethlas}. The first author would also like to thank Ruiqi Chen, Xuanyou Li, Yiyang Yang for helpful discussions.
\addtocontents{toc}{\protect\setcounter{tocdepth}{1}}
\section{Perfectoid torsors and geometric Sen theory}\label{sec:camargo}
In this section, we recall the notions of perfectoidness
for profinite \'etale torsors.
We then review the geometric Sen map and its basic properties
needed in the subsequent sections.
\subsection{Reminders on perfectoid towers}

Let $K/\mathbb Q_p$ be a complete nonarchimedean field.  We write
$\Rig_K$ for the category of separated rigid analytic spaces locally of
finite type over $\Spa(K,K^\circ)$.  Let $(\Lambda,\leq)$ be a small directed set with
a least element $0$. 

\begin{definition}
    A \emph{profinite \'etale tower} over $Y\in\Rig_K$
is a filtered inverse system
\[
  \mathcal Y=(Y_\lambda)_{\lambda\in\Lambda},
  \qquad Y_0=Y,
\]
such that, for $\lambda\leq\mu$, the transition morphism
$Y_\mu\to Y_\lambda$ is finite \'etale and surjective.
\end{definition}

Let $Y=\Spa(A_0,A_0^\circ)$ and
$Y_\lambda=\Spa(A_\lambda,A_\lambda^\circ)$ for all $\lambda\in\Lambda$
be a profinite \'etale tower.  For each $\lambda$, let $\rho_\lambda$
denote the spectral seminorm of $A_\lambda$.  The maximum-modulus formula
\cite[Theorems~1.2.1 and~1.3.1]{Berkovich1990} gives
\[
  \rho_\lambda(a)
    =\max_{x\in\mathcal M(A_\lambda)}|a(x)|
    \qquad(a\in A_\lambda).
\]
If $\lambda\leq\mu$,
the finite surjective map induces a surjection
$\mathcal M(A_\mu)\to\mathcal M(A_\lambda)$.  Hence
\[
  \rho_\lambda(a)
    =\max_{x\in\mathcal M(A_\lambda)}|a(x)|
    =\max_{x\in\mathcal M(A_\mu)}|a(x)|
    =\rho_\mu(a)
    \qquad(a\in A_\lambda).
\]
The spectral seminorms therefore induce a single seminorm $\rho$ on the
filtered colimit, and we put
\begin{equation}\label{eq:uniform-completion}
  A_{\mathrm{alg}}=\varinjlim_{\lambda\in\Lambda}A_\lambda,
  \qquad
  A_\infty^{\mathrm u}
    =\widehat{A_{\mathrm{alg}}}^{\,\rho},
\end{equation}
where the latter is the separated completion.
For the integral presentation, set
\begin{equation}\label{eq:integral-completion}
  A_{\mathrm{alg}}^+
    =\varinjlim_{\lambda\in\Lambda}A_\lambda^\circ,
  \qquad
  A_\infty^+
    =\varprojlim_{m\geq1}A_{\mathrm{alg}}^+/p^mA_{\mathrm{alg}}^+,
  \qquad
  A_\infty=A_\infty^+[1/p].
\end{equation}

\begin{definition}\label{def:four-perfectoid-notions}
Let $\mathcal Y=(Y_\lambda)_{\lambda\in\Lambda}$ be a profinite
\'etale tower, and put
\[
  \mathcal Y^\diamond
    =\varprojlim_{\lambda\in\Lambda}Y_\lambda^\diamond.
\]
\begin{enumerate}
\item Suppose that $Y$ is affinoid.  Write
  $Y_\lambda=\Spa(A_\lambda,A_\lambda^\circ)$.  The tower $\mathcal Y$
  is \emph{affinoid perfectoid} if the pair
  $(A_\infty,A_\infty^+)$ in \eqref{eq:integral-completion} is a
  perfectoid affinoid $K$-algebra in the sense of
  \cite[Definition~6.1.1 and Remark~7.1.3]{ScholzeWeinstein2020}.

\item Suppose that $Y$ is affinoid.  The tower $\mathcal Y$ is
  \emph{diamondian affinoid perfectoid} if $\mathcal Y^\diamond$ is
  represented by an affinoid perfectoid space in the sense of
  \cite[Definition~7.1.2 and Remark~7.1.3]{ScholzeWeinstein2020}.

\item The tower $\mathcal Y$ is \emph{perfectoid} if there are an index
  $\lambda\in\Lambda$ and an
  affinoid open covering
  \[
    Y_\lambda=\bigcup_{i\in I}V_i
  \]
  such that, for every $i\in I$, the tail tower
  \[
    \bigl(V_i\times_{Y_\lambda}Y_\mu\bigr)_{\mu\geq\lambda}
  \]
  is affinoid perfectoid.

\item The tower $\mathcal Y$ is \emph{diamondian perfectoid} if
  $\mathcal Y^\diamond$ is represented by a perfectoid space.
\end{enumerate}
\end{definition}

An affinoid perfectoid tower is perfectoid and diamondian affinoid
perfectoid; a perfectoid tower is diamondian perfectoid, and a diamondian
affinoid perfectoid tower is diamondian perfectoid.  These implications
follow from the affinoid construction following
\cite[Definition~4.3]{Scholze2013} and gluing.  Their converses do not
hold in general.

Finite pullback and an example separating the two affinoid notions
are discussed in Proposition~\ref{prop:finite-diamondian-pullback}
and Remark~\ref{rem:bchh-example}.
The following lemmas compare the finite-level completions used in
these definitions.

\begin{lemma}[The separated spectral completion ignores nilpotents]
\label{lem:reduction-completion}
Put $N_\lambda=\sqrt{(0)}\subset A_\lambda$ and
$A_{\lambda,\mathrm{red}}=A_\lambda/N_\lambda$.  Replacing every
$A_\lambda$ by $A_{\lambda,\mathrm{red}}$ does not change the separated
spectral completion:
\begin{equation}\label{eq:reduction-completion}
  \widehat{\varinjlim_{\lambda\in\Lambda}A_\lambda}^{\,\rho}
    \xrightarrow{\ \sim\ }
      \widehat{\varinjlim_{\lambda\in\Lambda}
        A_{\lambda,\mathrm{red}}}^{\,\rho}.
\end{equation}
The reduced transition maps remain finite \'etale and surjective.
\end{lemma}

\begin{proof}
For $\lambda\leq\mu$, finite \'etale base change commutes with reduction,
so
\[
  N_\mu=N_\lambda A_\mu,
  \qquad
  A_{\mu,\mathrm{red}}
   =A_\mu\otimes_{A_\lambda}A_{\lambda,\mathrm{red}}.
\]
Every point of $\Spa(A_\lambda,A_\lambda^\circ)$ annihilates
$N_\lambda$.  Consequently the spectral seminorm on $A_\lambda$ is the
pullback of the spectral seminorm on $A_{\lambda,\mathrm{red}}$.  Since
$A_{\lambda,\mathrm{red}}$ is a reduced affinoid
$K$-algebra, its spectral seminorm is a norm.  Hence
$\ker(\rho_\lambda)=N_\lambda$, and the quotient map identifies
$A_\lambda/\ker(\rho_\lambda)$ isometrically with
$A_{\lambda,\mathrm{red}}$.  Passing to the filtered colimit and then
taking the separated completion proves
\eqref{eq:reduction-completion}.
\end{proof}

\begin{lemma}[Comparison of the two completions]
\label{lem:integral-uniform-comparison}
For the maximal-plus presentation
$A_\lambda^+=A_\lambda^\circ$ fixed above, there are canonical
isomorphisms
\begin{equation}\label{eq:integral-uniform-comparison}
  A_\infty^+\xrightarrow{\ \sim\ }(A_\infty^{\mathrm u})^\circ,
  \qquad
  A_\infty\xrightarrow{\ \sim\ }A_\infty^{\mathrm u}.
\end{equation}
Moreover,
\[
  \ker(A_{\mathrm{alg}}^+\longrightarrow A_\infty^+)=\ker(\rho)
  =\ker(A_{\mathrm{alg}}\longrightarrow A_\infty^{\mathrm u}).
\]
\end{lemma}

\begin{proof}
Put $C=A_{\mathrm{alg}}$ and $C^+=A_{\mathrm{alg}}^+$.  In an arbitrary,
possibly nonreduced, affinoid $K$-algebra an element is power-bounded
exactly when its spectral seminorm is at most $1$
\cite[\S~6.2.3, Proposition~1]{BoschGuntzerRemmert1984}.  Since every
element of $C$ is represented at a finite level and the transition maps
preserve spectral seminorms, this gives
\begin{equation}\label{eq:unit-ball-colimit}
  C^+=\{a\in C:\rho(a)\leq1\}.
\end{equation}
For every $m\geq0$, division by $p^m$ in a finite-level $K$-algebra and
\eqref{eq:unit-ball-colimit} give
\begin{equation}\label{eq:p-adic-spectral-balls}
  p^mC^+=\{a\in C:\rho(a)\leq|p|^m\}.
\end{equation}
It follows first that
\[
  \bigcap_{m\geq0}p^mC^+=\ker(\rho),
\]
so the separated $p$-adic completion kills the kernel of the spectral
seminorm.  In particular, at a nonreduced finite level every nilpotent in
$C^+$ is infinitely $p$-divisible inside $C^+$ and disappears in this
completion.  It follows next from
\eqref{eq:p-adic-spectral-balls} that the $p$-adic topology on the
separated quotient of $C^+$ is the topology induced by $\rho$, because
the radii $|p|^m$ are cofinal among the positive radii.  Completing
therefore identifies $A_\infty^+$ with the closed unit ball in
$A_\infty^{\mathrm u}$.  The completed spectral norm is
power-multiplicative, so that closed unit ball is
$(A_\infty^{\mathrm u})^\circ$.  Finally, every element of
$A_\infty^{\mathrm u}$ becomes power-bounded after multiplication by a
sufficiently large power of $p$.  Inverting $p$ gives the second
isomorphism in \eqref{eq:integral-uniform-comparison}.
\end{proof}

For a tower satisfying Definition~\ref{def:four-perfectoid-notions}(1),
Lemma~\ref{lem:integral-uniform-comparison} identifies its defining
condition with the condition that $A_\infty^{\mathrm u}$ be a perfectoid
$K$-algebra, or equivalently that
\[
  \bigl(A_\infty^{\mathrm u},(A_\infty^{\mathrm u})^\circ\bigr)
\]
be a perfectoid affinoid $K$-algebra.  In this case
\[
  \widehat{\mathcal Y}
   =\Spa(A_\infty,A_\infty^+)
   =\Spa\bigl(A_\infty^{\mathrm u},
              (A_\infty^{\mathrm u})^\circ\bigr)
\]
is an affinoid perfectoid space.  This is the condition of
\cite[Definition~4.3(i)]{Scholze2013} for the displayed affinoid
presentation and the filtered version of
\cite[Definition~5.1.1]{KedlayaLiu2019}.

\subsection{The geometric Sen operators}\label{sec:sen-operators}

Let $X$ be smooth. We recall the intrinsic geometric Sen operator and
its torsor form from
\cite[Theorems~3.3.2 and~3.3.4]{RodriguezCamargo2026}.

\begin{theorem}\label{thm:geometric-sen}
Every finite locally free $\Oh_X$-module $\mathcal F$ on $X_{\proet}$
has a canonical Higgs field
\[
 \theta_{\mathcal F}:\mathcal F\longrightarrow
 \mathcal F\otimes_{\OO_X}\Omega^1_{X/\Cp}(-1),
 \qquad \theta_{\mathcal F}\wedge\theta_{\mathcal F}=0,
\]
functorial in $\mathcal F$ and compatible with pullback along morphisms
of smooth rigid spaces.

Let $\pi:\widetilde X\to X$ be a profinite \'etale torsor under a
compact $p$-adic Lie group $G$. Put $\mathfrak g=\Lie G$ and let
$\mathfrak g_\pi^\vee$ be the local system associated with the coadjoint
representation. There is a canonical morphism on $X_{\proet}$
\begin{equation}\label{eq:sen-map}
 \theta_{\widetilde X}:
 \Oh_X\otimes_{\Qp}\mathfrak g_\pi^\vee
 \longrightarrow
 \Oh_X\otimes_{\OO_X}\Omega^1_{X/\Cp}(-1).
\end{equation}
For a finite-dimensional locally analytic representation
$\rho:G\to\GL(V)$, put $\mathcal V=\Oh_X\otimes_{\Qp}V_\pi$.
The differentiated action, viewed as
$d\rho:\mathcal V\to\mathcal V\otimes_{\Qp}\mathfrak g_\pi^\vee$,
satisfies
\[
 \theta_{\mathcal V}
   =(\mathrm{id}_{\mathcal V}\otimes\theta_{\widetilde X})\circ d\rho.
\]
Such $\theta_{\widetilde X}$ is called the \emph{Sen map} of $\widetilde X$.
\end{theorem}

More explicitly, let $f:Y\to X$ be a morphism of smooth rigid spaces,
let $\varphi:H\to G$ be a homomorphism of compact $p$-adic Lie groups,
and let $\widetilde Y\to\widetilde X$ be a $\varphi$-equivariant
morphism over $f$, where $\pi_Y:\widetilde Y\to Y$ is an $H$-torsor.
Writing $\mathfrak h=\Lie H$, functoriality is the commutative diagram
\begin{equation}\label{eq:sen-functoriality}
\begin{CD}
 \Oh_Y\otimes_{\Qp}f^*\mathfrak g_\pi^\vee
 @>{f^*\theta_{\widetilde X}}>>
 \Oh_Y\otimes_{\OO_Y}f^*\Omega^1_{X/\Cp}(-1)\\
 @V{\mathrm{id}\otimes d\varphi^\vee}VV
 @VV{\mathrm{id}\otimes df}V\\
 \Oh_Y\otimes_{\Qp}\mathfrak h^\vee_{\pi_Y}
 @>{\theta_{\widetilde Y}}>>
 \Oh_Y\otimes_{\OO_Y}\Omega^1_{Y/\Cp}(-1).
\end{CD}
\end{equation}
This is \cite[Equation~(3.12)]{RodriguezCamargo2026}.

For the local construction, let $X=\Spa(A,A^\circ)$ be a small
smooth affinoid of dimension $d$ and fix a toric chart. Set
\[
\begin{split}
 X_n={}&X\times_{\TT^d_{\Cp}}
 \Spa\bigl(\Cp\langle T_1^{\pm1/p^n},\ldots,T_d^{\pm1/p^n}\rangle,
 \Cp^\circ\langle T_1^{\pm1/p^n},\ldots,T_d^{\pm1/p^n}\rangle\bigr),\\
 A_n={}&\OO(X_n),\qquad
 A_{\mathrm{tor}}=\widehat{\varinjlim_n A_n}.
\end{split}
\]
The map to the torus sends $T_i$ to $(T_i^{1/p^n})^{p^n}$.
The toric tower defines the affinoid perfectoid space
$X_\infty=\Spa(A_{\mathrm{tor}},A_{\mathrm{tor}}^\circ)$
\cite[Section~4]{Scholze2013}. Its Galois group is $\Gamma=\Zp^d$.
Let $\eta_1,\ldots,\eta_d$ be the standard basis of $\Lie\Gamma$.
The toric Sen isomorphism identifies its dual basis with
\[
 \omega_j=\theta_{X_\infty}(\eta_j^\vee)
 \quad\text{in}\quad
 A_{\mathrm{tor}}\otimes_A\Omega^1_{A/\Cp}(-1).
\]

Here is the finite-level construction of these operators.
For a finite locally free $\Oh_X$-module $\mathcal F$, put
$M=\mathcal F(X_\infty)$. For this local construction, make the finite
\'etale refinement used in relative Sen theory and retain the notation
for the resulting chart and tower. There is a finite projective $A_n$-module
$P_n\subset M$ such that
\[
 A_{\mathrm{tor}}\otimes_{A_n}P_n\simeq M,\qquad
 M^{\Gamma\text{-}\mathrm{la}}
   =\varinjlim_{m\ge n}(A_m\otimes_{A_n}P_n);
\]
see \cite[Proposition~2.2.14, Theorem~2.4.4, and
Remark~3.2.2]{RodriguezCamargo2026}.
Differentiating the action of a sufficiently small open subgroup of
$\Gamma$ on $P_n$ gives
\[
 \nabla_j=d\eta_j\in\End_{A_n}(P_n),\qquad
 [\nabla_i,\nabla_j]=0.
\]
The operators are $A_n$-linear because the action on $A_n$ factors
through a finite quotient. Extending them $A_{\mathrm{tor}}$-linearly
defines
\[
 \theta_{\mathcal F}(a\otimes v)
  =\sum_j(a\otimes\nabla_jv)\otimes\omega_j.
\]
These local Higgs fields descend and are independent of the chart
\cite[Proposition~3.2.3 and Theorem~3.3.2]{RodriguezCamargo2026}.
Applying the construction to locally analytic functions on $G$, the
resulting derivations commute with left translation and hence factor
through $\Lie G$. Dualizing gives the local torsor map
\eqref{eq:sen-map}
\cite[Proposition~3.2.4]{RodriguezCamargo2026}.

\section{A perfectoid torsor with nonsurjective Sen map}
\label{sec:construction}

In this section, we will prove Theorem \ref{thm:counterexample}. 
\subsection{Pullback along universally injective morphisms} Let $K/\Qp$ be a complete nonarchimedean field. 
The most important observation during our construction is that universal pull-back protects perfectoidness.
\begin{definition}\label{def:universally-injective}
Let $f\colon X\to Y$ be a morphism in $\Rig_K$.
\begin{enumerate}
\item The morphism $f$ is a \emph{universal homeomorphism} if, for every
  morphism $T\to Y$ in $\Rig_K$, the map
  \[
    \lvert X\times_YT\rvert\longrightarrow\lvert T\rvert
  \]
  of underlying topological spaces is a homeomorphism.

\item The morphism $f$ is \emph{universally injective} if, for every
  morphism $T\to Y$ in $\Rig_K$, the same map is injective.
\end{enumerate}
\end{definition}

\begin{proposition}
\label{prop:ui-rigid-criterion}
Let $f\colon X\to Y$ be a morphism in $\Rig_K$.  Then $f$ is
universally injective if and only if, for every morphism $T\to Y$ in
$\Rig_K$, the map
\[
  (X\times_YT)_{\mathrm{rig}}\longrightarrow T_{\mathrm{rig}}
\]
is injective.
\end{proposition}

\begin{proof}
Only the converse requires proof.  Fix $T\to Y$, put
$W=X\times_YT$, and let $p_1,p_2$ be the two projections from
$P=W\times_TW$.  Since $f$ is separated, the diagonal
\[
  \Delta\colon W\longrightarrow P
\]
is a closed immersion.  Apply the hypothesis after the base change
$W\to Y$.  Under the canonical identification
\[
  X\times_YW\simeq W\times_TW=P,
\]
the resulting map on rigid points is $p_2\colon P_{\mathrm{rig}}
\to W_{\mathrm{rig}}$, and it is injective.  If $z\in P_{\mathrm{rig}}$,
then both $z$ and $\Delta(p_2(z))$ are rigid points with the same image
under $p_2$.  Hence
\[
  z=\Delta(p_2(z)).
\]
Thus the closed subset $\Delta(W)\subseteq P$ contains every rigid point
of $P$.  It equals $P$: otherwise its open complement would contain a
nonempty affinoid open, and the affinoid Nullstellensatz would give a rigid
point in that complement.

The pointwise description in the construction of fibre products
\cite[Proposition~1.2.2]{HuberEtale} gives a surjection
\[
  |W\times_TW|\longrightarrow |W|\times_{|T|}|W|.
\]
If $w_1,w_2\in W$ have the same image in $T$, choose $z\in P$ mapping to
$(w_1,w_2)$.  Since $\Delta(W)=P$, the point $z$ lies on the diagonal, so
$w_1=w_2$.  Thus $|W|\to|T|$ is injective.  As $T\to Y$ was arbitrary,
$f$ is universally injective.
\end{proof}

\begin{remark}
The analogous rigid-point criterion for universal homeomorphisms is false.
Over $\Cp$, let
\[
  D=\Spa(\Cp\langle T\rangle,\Cp^\circ\langle T\rangle),
  \qquad
  U=\bigcup_{a\in\Cp^\circ/\mathbb{C}_p^{\circ\circ}}\bigcup_{n\geq 1}\{x\in D:|T-a|_x\leq p^{-\frac{1}{n}}\}
\]
The open immersion $U\hookrightarrow D$ induces a bijection on rigid points
after every base change in $\Rig_{\Cp}$, but it omits the Gauss point and is
therefore not a universal homeomorphism.
\end{remark}

\begin{remark}
Let $A\to B$ be a finite morphism of $K$-affinoid algebras.  Then
\[
  \Spa(B,B^\circ)\longrightarrow\Spa(A,A^\circ)
\]
is a universal homeomorphism in $\Rig_K$ if and only if the schematic morphism
\[
  \Spec(B)\longrightarrow\Spec(A)
\]
is a finite universal homeomorphism, by
\cite[Proposition~2.3.2]{GuoHT}.
\end{remark}

\begin{lemma}\label{lem:finite-factor}
Let $g\colon U\to V$ be a finite universally injective morphism between
affinoid objects of $\Rig_K$.  If $Z\hookrightarrow V$ is the
closed analytic image of $g$, then the induced morphism $U\to Z$ is a
finite universal homeomorphism.
\end{lemma}

\begin{proof}
Take $V=\Spa(A,A^\circ)$ and $U=\Spa(B,B^\circ)$, and put
\[
  I=\ker(A\longrightarrow B),
  \qquad
  Z=\Spa(A/I,(A/I)^\circ).
\]
The ideal $I$ is closed because $A$ is noetherian.  The homomorphism
$A/I\hookrightarrow B$ is finite, and
$\Spec(B)\to\Spec(A/I)$ is integral and surjective.  Since $Z\to V$ is a
monomorphism, for every $T\to Z$ there is a canonical identification
\[
  U\times_ZT\simeq U\times_VT.
\]
Thus $U\to Z$ is universally injective.  It is also universally
surjective.  Indeed, let $T\to Z$ and $t\in T_{\mathrm{rig}}$.  On an
affinoid neighbourhood $T=\Spa(C,C^\circ)$ of $t$, finiteness gives
\[
  B\otimes_{A/I}C\xrightarrow{\ \sim\ }
  B\widehat\otimes_{A/I}C
\]
and lying over gives a maximal ideal above the maximal ideal defined by $t$;
its residue field is finite over $\kappa(t)$, so it defines a rigid point
of $U\times_ZT$ above $t$.  Thus the image of $U\times_ZT\to T$ contains
all rigid points.  This image is closed because the map is finite, and
rigid points are dense, so the map is surjective.  It is also injective by
the preceding paragraph, hence a closed continuous bijection and therefore
a homeomorphism.
\end{proof}

\begin{proposition}[Local factorization of universally injective morphisms]
\label{prop:ui-factorization}
Let $K/\mathbb Q_p$ be a complete nonarchimedean field and let $f\colon X\to Y$ be a
universally injective morphism in $\Rig_K$.  There exist an admissible open
covering $X=\bigcup_iU_i$, admissible open subspaces $V_i\subseteq Y$, and
factorizations
\[
  U_i\xrightarrow{q_i}Z_i\xrightarrow{j_i}V_i
\]
of $f|_{U_i}$ such that $q_i$ is a finite universal homeomorphism and
$j_i$ is a closed immersion.
\end{proposition}

\begin{proof}
It suffices to show that every adic point $x\in X$ admits
affinoid open neighbourhoods $U$ of $x$ and $V$ of $f(x)$
such that $U\to V$ is finite.
We may initially restrict to affinoid neighbourhoods
\[
    X=\operatorname{Spa}(B,B^\circ),\qquad
    Y=\operatorname{Spa}(A,A^\circ).
\]
Put $y=f(x)$, and let $\xi$ and $\eta$ be the maximal rank-one
generizations of $x$ and $y$, respectively.

The closed diagonal of $f$ is surjective on underlying spaces,
hence is a nilpotent closed immersion. This property persists
under complete field extensions. By the affinoid Nullstellensatz,
the geometric fibres are therefore zero-dimensional with at
most one point. Thus $\mathcal H(\xi)/\mathcal H(\eta)$ is
finite and purely inseparable, and characteristic zero gives
\(
    \mathcal H(\xi)=\mathcal H(\eta).
\)

Choose a surjection
\[
    A\langle T_1,\ldots,T_n\rangle\twoheadrightarrow B,
    \qquad T_i\longmapsto b_i,
\]
with $b_i\in B^\circ$, and fix $\varpi\in K$ with
$0<|\varpi|<1$.
By density of $\operatorname{Frac}(A/\ker\eta)$ in
$\mathcal H(\eta)$, choose fractions $a_i=u_i/v_i$, with
$v_i(\eta)\ne0$, such that
\[
    |b_i-a_i|(\xi)<|\varpi|.
\]
These strict inequalities persist at $x$, so $|a_i(y)|\le1$.
After a rational localization around $y$, imposing
\[
    |u_i|\le|v_i|,\qquad |\varpi|^{m_i}\le|v_i|
\]
for sufficiently large $m_i$, and the corresponding base
change on $X$, we may assume $a_i\in A^\circ$.
The elements $c_i=b_i-a_i$ then satisfy
\[
    A\langle S_1,\ldots,S_n\rangle\twoheadrightarrow B,
    \qquad S_i\longmapsto c_i,
    \qquad |c_i(\xi)|<|\varpi|.
\]

Set
\[
    D=\bigcup_{i=1}^n\{z\in X:|c_i(z)|\ge|\varpi|\}.
\]
This is a finite union of rational subsets. Since $f$ is
spectral, $f(D)$ is closed in the constructible topology
\cite[Tag~0A2S]{Stacks}.
It contains no generization of $y$: otherwise, for some
$z\in D$, the maximal rank-one generization $\zeta$ of $z$
would satisfy
\[
    f(\zeta)=\eta=f(\xi).
\]
Injectivity would give $\zeta=\xi$, contradicting
$|c_i(\xi)|<|\varpi|$ for every $i$.
It follows from \cite[Tag~0903]{Stacks} that
$y\notin\overline{f(D)}$.
Choose a rational neighbourhood $V$ of $y$ disjoint from
$\overline{f(D)}$, and put
\[
    U=X\times_YV,\qquad
    A_V=\mathcal O(V),\qquad B_V=\mathcal O(U).
\]
Then $x\in U$, and $|c_i(z)|<|\varpi|$ for every $z\in U$.
Hence
\[
    c_i/\varpi\in B_V^\circ,\qquad
    c_i\in B_V^{\circ\circ},
\]
by \cite[Proposition~7.52(1)]{Wedhorn}.

Choose a complete ring of definition $A_0\subset A_V$
containing $\varpi$, and let $B_0$ be the image of
$A_0\langle S_1,\ldots,S_n\rangle$ in $B_V$.
The surjection onto $B_V$ is strict, so $B_0$ is a ring of
definition. Each $c_i$ is topologically nilpotent, and hence
its image in $B_0/\varpi B_0$ is nilpotent.
These images generate $B_0/\varpi B_0$ over
$A_0/\varpi A_0$, so this algebra is finite.
Topological Nakayama \cite[Tag~031D]{Stacks} now implies
that $B_0$ is finite over $A_0$.
Inverting $\varpi$ shows that $B_V$ is finite over $A_V$,
as required.

The neighbourhoods $U$ obtained for all $x\in X$ form an
admissible open covering. Applying Lemma~3.5 to each finite,
universally injective morphism $U\to V$ gives the asserted
factorization.
\end{proof}


The next lemma is the algebraic step used after the local factorization.  Its
proof identifies the completion of the finite-level rings.

\begin{lemma}[Uniform completion after pullback along a finite universal
homeomorphism]
\label{lem:uh-completion}
Let 
\[
  h\colon U=\Spa(B,B^\circ)\longrightarrow
  Z=\Spa(R,R^\circ)
\]
be a finite universal homeomorphism of affinoid rigid $K$-spaces.  Let
$\mathcal Z=(Z_\lambda)_{\lambda\in\Lambda}$, with
$Z_\lambda=\Spa(R_\lambda,R_\lambda^\circ)$, be a profinite \'etale
tower over $Z$, and put
\[
  S=\widehat{\varinjlim_{\lambda\in\Lambda}R_\lambda}^{\,\rho},
  \qquad
  B_\lambda=R_\lambda\widehat\otimes_RB
           =R_\lambda\otimes_RB,
  \qquad
  U_\lambda=U\times_ZZ_\lambda
           =\Spa(B_\lambda,B_\lambda^\circ).
\]
Assume that $S$ is a perfectoid $K$-algebra.  Then the rings
$B_{\lambda,\mathrm{red}}$ admit compatible isometric embeddings
into $S$, and the finite-level rings induce an isomorphism
\begin{equation}\label{eq:uh-uniform-completion}
  \widehat{\varinjlim_{\lambda\in\Lambda}B_\lambda}^{\,\rho}
    \xrightarrow{\ \sim\ }S.
\end{equation}
In particular, $(U_\lambda)_{\lambda\in\Lambda}$ is affinoid perfectoid
and the finite-level rings of the two towers have the same completion.
\end{lemma}

\begin{proof}
Finite \'etale base change commutes with reduction.  Thus
\[
  R_{\lambda,\mathrm{red}}
    =R_\lambda\otimes_RR_{\mathrm{red}},
  \qquad
  B_{\lambda,\mathrm{red}}
    =R_{\lambda,\mathrm{red}}
       \otimes_{R_{\mathrm{red}}}B_{\mathrm{red}}.
\]
The ring on the right is finite \'etale over the reduced ring
$B_{\mathrm{red}}$, hence is reduced; it is therefore the reduction of
$B_\lambda$.  The induced morphism
\[
  \Spa(B_{\mathrm{red}})\longrightarrow\Spa(R_{\mathrm{red}})
\]
is again a finite universal homeomorphism.  By
Lemma~\ref{lem:reduction-completion}, the systems $(R_\lambda)_\lambda$
and $(B_\lambda)_\lambda$ have the same separated spectral completions as
their levelwise reductions.  In particular, the reduced $R$-side
completion is canonically $S$.  We may therefore replace $h$ and the two
systems by their reductions and assume from now on that all displayed
rings are reduced.

For every $\lambda\in\Lambda$, the square
\[
\begin{CD}
  U_\lambda @>>> Z_\lambda \\
  @VVV @VVV \\
  U @>{h}>> Z
\end{CD}
\]
is Cartesian; in particular, $R_\lambda\to B_\lambda$ is obtained from
$R\to B$ by base change.  The morphism $U_\lambda\to Z_\lambda$ is
therefore a finite universal homeomorphism.
By \cite[Proposition~2.3.2]{GuoHT}, the induced morphism
\[
  \Spec(B_\lambda)\longrightarrow\Spec(R_\lambda)
\]
is a universal homeomorphism of schemes.  Since both rings are reduced
and contain $\mathbb Q$, \cite[Corollary~1.4.7]{KedlayaLiu2019} gives
\begin{equation}\label{eq:sn-finite-level}
  R_\lambda^{\mathrm{sn}}\xrightarrow{\ \sim\ }
  B_\lambda^{\mathrm{sn}}.
\end{equation}

The algebra $S$ is perfectoid, hence seminormal by
\cite[Theorem~3.7.4]{KedlayaLiu2019}.  The map $R_\lambda\to S$
consequently factors uniquely through $R_\lambda^{\mathrm{sn}}$.  The
composite
\[
  B_\lambda\longrightarrow B_\lambda^{\mathrm{sn}}
  \xrightarrow[\sim]{\,\eqref{eq:sn-finite-level}^{-1}\,}
  R_\lambda^{\mathrm{sn}}\longrightarrow S
\]
is a homomorphism
\begin{equation}\label{eq:jlambda}
  \jmath_\lambda\colon B_\lambda\longrightarrow S
\end{equation}
whose restriction to $R_\lambda$ is the canonical map.  If
$\lambda\leq\mu$, functoriality and uniqueness of the factorization
through seminormalization give
\begin{equation}\label{eq:jlambda-compatible}
  \jmath_\mu\circ(B_\lambda\longrightarrow B_\mu)=\jmath_\lambda.
\end{equation}
The map \eqref{eq:jlambda} is bounded.  Indeed, choose generators
$e_1,\ldots,e_r$ of the finite $R_\lambda$-module $B_\lambda$.  For
$b=\sum_i a_ie_i$ one has
\[
  \|\jmath_\lambda(b)\|_S
  \leq\max_i\bigl(\|a_i\|_{R_\lambda}\,
                    \|\jmath_\lambda(e_i)\|_S\bigr),
\]
which bounds \eqref{eq:jlambda} with respect to a finite-module quotient norm
on $B_\lambda$.  Such a quotient norm induces the given affinoid topology,
because a finite morphism of affinoid algebras is strict.
If $b\in B_\lambda^\circ$, the set $\{b^m:m\geq0\}$ is bounded;
boundedness of $\jmath_\lambda$ then shows that
$\{\jmath_\lambda(b)^m:m\geq0\}$ is bounded in $S$.  Thus
$\jmath_\lambda(B_\lambda^\circ)\subseteq S^\circ$.

We next prove that the map of Berkovich spectra
$\mathcal M(S)\to\mathcal M(R_\lambda)$ is surjective.  Fix
$x_\lambda\in\mathcal M(R_\lambda)$.  For $\mu\geq\lambda$, let
\[
  F_\mu=
   \{x_\mu\in\mathcal M(R_\mu):x_\mu|_{R_\lambda}=x_\lambda\}.
\]
Each $F_\mu$ is a nonempty finite discrete space.  Consider the compact
space $\prod_{\mu\geq\lambda}F_\mu$.  For every
$\lambda\leq\mu\leq\nu$, impose the closed condition that the
$\nu$-component restrict to the $\mu$-component.  Any finite collection
of these conditions has a solution: choose an upper bound $\tau$ for all
indices occurring in it, choose a point of $F_\tau$, and restrict that
point to the finitely many components.  The finite intersection property
therefore gives
\begin{equation}\label{eq:filtered-fibres}
  \varprojlim_{\mu\geq\lambda}F_\mu\neq\varnothing.
\end{equation}
Since $\Lambda_{\geq\lambda}$ is cofinal in $\Lambda$, an element of this
inverse limit defines a multiplicative seminorm $|\cdot|$ on
$\varinjlim_{\mu\in\Lambda}R_\mu$ satisfying
$|a|\leq\rho(a)$.  It therefore extends uniquely and continuously to
$S$, and the extension lies in $\mathcal M(S)$ above $x_\lambda$.

The surjection $\mathcal M(S)\to\mathcal M(R_\lambda)$ factors through
$\mathcal M(B_\lambda)$ by \eqref{eq:jlambda}.  Since
$\mathcal M(B_\lambda)\to\mathcal M(R_\lambda)$ is a homeomorphism,
surjectivity of
the composite implies that
\[
  \mathcal M(S)\longrightarrow\mathcal M(B_\lambda)
\]
is surjective.  Therefore, for every $b\in B_\lambda$,
\begin{equation}\label{eq:norm-identity}
\begin{aligned}
  \rho_{B_\lambda}(b)
  &=\sup_{x\in\mathcal M(B_\lambda)}|b(x)| \\
  &=\sup_{y\in\mathcal M(S)}|\jmath_\lambda(b)(y)|
   =\rho_S(\jmath_\lambda(b)).
\end{aligned}
\end{equation}
Since $B_\lambda$ is reduced, $\rho_{B_\lambda}$ is a norm.  Thus
$\jmath_\lambda$ is an isometric embedding.

Set
\[
  D=\operatorname{im}\!\left(
       \varinjlim_{\lambda\in\Lambda}B_\lambda\longrightarrow S
     \right).
\]
The inclusions $R_\lambda\subseteq B_\lambda$ give
\[
  \operatorname{im}\!\left(
    \varinjlim_{\lambda\in\Lambda}R_\lambda\longrightarrow S
  \right)
  \subseteq D\subseteq S.
\]
The left-hand ring is dense in $S$ by the definition of $S$; hence $D$
is dense in $S$.  Equation~\eqref{eq:norm-identity} identifies the separated
normed colimit of the $B_\lambda$ with $D$ endowed with the restriction of
$\rho_S$.  Since $S$ is complete, the inclusion extends uniquely to an
isometric isomorphism
\[
  \widehat{\varinjlim_{\lambda\in\Lambda}B_\lambda}^{\,\rho}
    \xrightarrow{\ \sim\ }S.
\]
This proves \eqref{eq:uh-uniform-completion} for the reduced tower.
Lemma~\ref{lem:reduction-completion} gives the same equality for the
original finite-level rings.
\end{proof}

\begin{theorem}[Universally injective pullback]
\label{thm:ui-pullback}
Let $K/\mathbb Q_p$ be a perfectoid field and let
$f\colon X\to Y$ be a universally injective morphism in $\Rig_K$.  The
pullback of a perfectoid profinite \'etale tower over $Y$ is a perfectoid
profinite \'etale tower over $X$.  If the original tower is a $G$-torsor,
its pullback is a $G$-torsor.  If, moreover, $X$ and $Y$ are affinoid,
$f$ is finite, and the original tower is affinoid perfectoid, then the
pullback tower is affinoid perfectoid.
\end{theorem}

\begin{proof}
We first prove the following local assertion.  Suppose that $Y$ is
affinoid and that the given presentation of $\mathcal Y$ is affinoid
perfectoid.  Then its pullback along a universally injective
$f\colon X\to Y$ is perfectoid.  If in addition $X$ is affinoid and $f$
is finite, the pulled-back tower is affinoid perfectoid.

Put $X_\lambda=X\times_YY_\lambda$.  Finite \'etale base change commutes
with reduction, so there are Cartesian identifications
\begin{equation}\label{eq:reduced-pullback-tower}
  (X_\lambda)_{\mathrm{red}}
   =X_{\mathrm{red}}\times_{Y_{\mathrm{red}}}
      (Y_\lambda)_{\mathrm{red}}.
\end{equation}
The morphism $f_{\mathrm{red}}\colon X_{\mathrm{red}}\to
Y_{\mathrm{red}}$ is universally injective.  Indeed, for every
$T\to Y_{\mathrm{red}}$, the natural morphism
\[
  X_{\mathrm{red}}\times_{Y_{\mathrm{red}}}T
  \longrightarrow X\times_YT
\]
is a nilpotent closed immersion and hence a homeomorphism on underlying
topological spaces; injectivity follows from that of the base change of
$f$.  On every affinoid open,
Lemmas~\ref{lem:reduction-completion} and
\ref{lem:integral-uniform-comparison} identify the perfectoid condition
for a tower with that of its levelwise reduction.  We may therefore
replace $f$ and the two towers by their reductions and assume that $X$
and $Y$ are reduced.

Apply Proposition~\ref{prop:ui-factorization}.  Because $Y$ is affinoid,
we may refine its covering so that every target $V_i$ is a rational
affinoid subdomain of $Y$.  The tower restricted to $V_i$ is affinoid
perfectoid by \cite[Lemma~4.5(i)]{Scholze2013}.  It is enough to work with
one factorization
\[
  U=\Spa(B,B^\circ)\xrightarrow{q}
  Z=\Spa(R,R^\circ)\xrightarrow{j}
  V=\Spa(A,A^\circ),
\]
where $q$ is a finite universal homeomorphism and $j$ is a closed
immersion.  The ring $B$ is reduced.  If $I=\ker(A\to B)$, then
$R=A/I$ embeds into $B$ and is also reduced.

Write the rings of the restricted tower over $V$ as $A_\lambda$.  Its
pullback to $Z$ has rings
\begin{equation}\label{eq:closed-pullback-rings}
  R_\lambda=A_\lambda\widehat\otimes_AR=A_\lambda/IA_\lambda.
\end{equation}
The quotient $A\twoheadrightarrow R$ is strict and surjective.  Hence
\cite[Proposition~5.1.3(f)]{KedlayaLiu2019}, applied to
\eqref{eq:closed-pullback-rings}, gives an affinoid perfectoid finite
\'etale tower $(R_\lambda)_{\lambda\in\Lambda}$ and a perfectoid
$K$-algebra
\[
  S=\widehat{\varinjlim_{\lambda\in\Lambda}R_\lambda}^{\,\rho}.
\]
Kedlaya--Liu state the proposition for a sequence; its proof uses only
the completed direct limit and therefore applies verbatim to the present
directed system.

The pullback from $Z$ to $U$ has rings
\begin{equation}\label{eq:uh-pullback-rings}
  B_\lambda=R_\lambda\widehat\otimes_RB
     =A_\lambda\widehat\otimes_AB.
\end{equation}
Lemma~\ref{lem:uh-completion}, applied to $q$, gives
\[
  \widehat{\varinjlim_{\lambda\in\Lambda}B_\lambda}^{\,\rho}=S.
\]
Thus the actual tower \eqref{eq:uh-pullback-rings} over $U$ is affinoid
perfectoid.  The sets $U_i$ cover $X$, proving the first part of the local
assertion.

If $X$ is affinoid and $f$ is finite, Lemma~\ref{lem:finite-factor}
instead gives one global factorization
\[
  X\longrightarrow Z\longrightarrow Y
\]
with $X\to Z$ a finite universal homeomorphism and $Z\to Y$ a closed
immersion.  The same two completed-algebra calculations, now made
globally, show that the uniform completion of the pullback tower is $S$.
It is therefore affinoid perfectoid.  This proves the local assertion.

We now treat a general perfectoid tower over an arbitrary $Y$.  By
Definition~\ref{def:four-perfectoid-notions}(3), choose
$\lambda\in\Lambda$ and an affinoid open covering
$Y_\lambda=\bigcup_{a\in I}V_a$ such that the tail tower over every
$V_a$ is affinoid perfectoid.  Set
\[
  W_a=(X\times_YY_\lambda)\times_{Y_\lambda}V_a.
\]
The morphism $W_a\to V_a$ is a base change of $f$ and is therefore
universally injective.  Applied to the pulled-back tail, the construction
in the local assertion gives an affinoid open covering
$W_a=\bigcup_jU_{a,j}$ for which every restricted tail tower over
$U_{a,j}$ is affinoid perfectoid.  The opens $U_{a,j}$ cover
$X\times_YY_\lambda$, so Definition~\ref{def:four-perfectoid-notions}(3)
shows that the pulled-back tower is perfectoid.
At every quotient level, a finite $G/H$-torsor remains a $G/H$-torsor
after base change, proving the assertion about $G$.  Under the additional
affinoid and finite hypotheses in the theorem, the final clause follows
directly from the affinoid part of the local assertion.
\end{proof}

\subsection{Construction of the torsor}\label{subsec: construction counter ex}

Choose $\varpi\in\Cp$ with $0<|\varpi|<1$, and let
\[
  \TT^2=\Spa\Cp\langle T_1^{\pm1},T_2^{\pm1}\rangle.
\]
Consider the rational subdomain
\[
  Y=\bigl\{|T_1-1|\leq|\varpi|,\ |T_2-1|\leq|\varpi|\bigr\}
   \subset\TT^2.
\]
Writing
\[
  T_1=1+\varpi u,
  \qquad
  T_2=1+\varpi v,
\]
identifies $Y$ with $\Spa\Cp\langle u,v\rangle$.

For every $n\geq0$, put
\[
  \TT_n^2
  =\Spa\Cp\langle
      T_1^{\pm1/p^n},T_2^{\pm1/p^n}
    \rangle
\]
and set
\[
  Y_n=Y\times_{\TT^2}\TT_n^2,
\]
where $\TT_n^2\to\TT^2$ is induced by
$T_i\mapsto(T_i^{1/p^n})^{p^n}$ for $i=1,2$.  The transition morphism
$Y_{n+1}\to Y_n$ is induced by
\[
  T_i^{1/p^n}\longmapsto
  \bigl(T_i^{1/p^{n+1}}\bigr)^p
  \qquad(i=1,2).
\]
Thus $Y_n\to Y$ is a finite \'etale torsor under
$(\mu_{p^n})^2$.

Choose compatible primitive $p^n$-th roots of unity in $\Cp$.  They
identify the group acting on the tower with
\[
  \varprojlim_n(\mu_{p^n})^2(\Cp)
    \simeq\mathbb Z_p^2=:G.
\]
Hence $(Y_n)_{n\geq0}$ is a profinite \'etale $G$-torsor over $Y$.

The standard toric tower $(\TT_n^2)_{n\geq0}$ is affinoid perfectoid by
\cite[Example~4.4]{Scholze2013}; let $\widetilde{\TT}^{\,2}$ denote its
associated affinoid perfectoid space.  Define the rational subdomain
\[
  \widetilde Y
  =\bigl\{|T_1-1|\leq|\varpi|,\ |T_2-1|\leq|\varpi|\bigr\}
  \subset\widetilde{\TT}^{\,2}.
\]
By \cite[Lemma~4.5(i)]{Scholze2013}, this is the affinoid perfectoid
space associated with $(Y_n)_{n\geq0}$ and
\[
  \widetilde Y\sim\varprojlim_nY_n.
\]
Write $\pi_Y\colon\widetilde Y\to Y$ for the resulting map.  Thus
$(Y_n)_{n\geq0}$ is an affinoid perfectoid profinite \'etale $G$-torsor
over $Y$.

Let
\[
  Z=V(v^2-u^3)\subset Y.
\]
The cusp $Z$ is not seminormal.  Its seminormalization, which here equals
its normalization, is
\begin{equation}\label{eq:cusp-normalization}
  \nu\colon X=\Spa\Cp\langle t\rangle\longrightarrow Z,
  \qquad
  u\longmapsto t^2,
  \qquad
  v\longmapsto t^3.
\end{equation}
Put $A=\Cp\langle t^2,t^3\rangle$ and $B=\Cp\langle t\rangle$.  The element
$t$ satisfies $T^2-t^2\in A[T]$, and $B=A+At$; hence $A\to B$ is finite.
Over $D(t^2)\subset\Spec(A)$ its inverse is given by $t=t^3/t^2$.  The
complement is the point $(t^2,t^3)$, whose inverse image is the single
point $(t)$, and both residue fields equal $\Cp$.  Thus
$\Spec(B)\to\Spec(A)$ is bijective and induces isomorphisms on all residue
fields.  It is subintegral by
\cite[Lemma~1.4.5]{KedlayaLiu2019}.  Since $B$ is normal and
$t^2,t^3\in A$ whereas $t\notin A$, it is the proper seminormalization of
$A$.  By \cite[Proposition~2.3.2]{GuoHT}, $\nu$ is a finite universal
homeomorphism of rigid spaces.  The space $X$ is a smooth rigid analytic
curve.  If $\iota\colon Z\hookrightarrow Y$ denotes the closed immersion,
set
\[
  f=\iota\circ\nu\colon X\longrightarrow Y.
\]
For every $n$, define
\[
  X_n=X\times_YY_n=\Spa(B_n,B_n^\circ),
  \qquad
  B_n=\Gamma(X_n,\OO_{X_n}).
\]
Thus $X_n\to X$ is a finite \'etale $(\mu_{p^n})^2$-torsor.

The morphism $\nu$ is a universal homeomorphism and $\iota$ is a closed
immersion.  Both are universally injective, so their composite $f$ is
universally injective.  Moreover, $f$ is finite.  The preceding
construction and Theorem~\ref{thm:ui-pullback} therefore show that
$(X_n)_{n\geq0}$ is an affinoid perfectoid profinite \'etale $G$-torsor.
Its associated affinoid perfectoid space is
\[
  \widetilde X=\Spa(B_\infty,B_\infty^\circ),
  \qquad
  B_\infty=\widehat{\varinjlim_nB_n}^{\,\rho}.
\]
It satisfies $\widetilde X\sim\varprojlim_nX_n$; write
$\pi_X\colon\widetilde X\to X$ for the resulting map.

\subsection{Failure of Sen surjectivity}
\label{sec:sen-proof}
Now we can prove that an affinoid perfectoid torsor can have non-surjective Sen map.

\begin{theorem}\label{thm: perfd not surjective}
    The $\mathbb Z_p^2$-torsor $\pi_X:\widetilde X\to X$ in Subsection \ref{subsec: construction counter ex} is affinoid perfectoid but the Sen map is not surjective at $t=0$.
\end{theorem}

\begin{proof}
It is already proved that $\pi_X:\widetilde X\to X$ is affinoid perfectoid in Subsection \ref{subsec: construction counter ex}. We only need to show that the Sen map is not surjective at $t=0$.

For every $n\geq0$, the definition $X_n=X\times_YY_n$ gives a Cartesian
square
\[
\begin{CD}
  X_n @>{f_n}>> Y_n \\
  @V{\pi_{X,n}}VV @VV{\pi_{Y,n}}V \\
  X @>{f}>> Y.
\end{CD}
\]
These squares commute with the transition morphisms.  Passing to the
perfectoid tilde-limits gives the $G$-equivariant commutative square
\[
\begin{CD}
  \widetilde X @>{\widetilde f}>> \widetilde Y \\
  @V{\pi_X}VV @VV{\pi_Y}V \\
  X @>{f}>> Y.
\end{CD}
\]
Both horizontal maps are induced by $u=t^2$ and $v=t^3$, and the group
homomorphism $G\to G$ attached to this diagram is the identity.

The map on differentials is
\begin{equation}\label{eq:differential}
  df\colon f^*\Omega^1_{Y/\Cp}\longrightarrow\Omega^1_{X/\Cp},
  \qquad
  du\longmapsto2t\,dt,
  \qquad
  dv\longmapsto3t^2\,dt.
\end{equation}
In particular,
\begin{equation}\label{eq:image-df}
  \im(df)=t\OO_X\,dt.
\end{equation}

Let $f_{\proet}\colon X_{\proet}\to Y_{\proet}$ be the induced morphism of
pro-\'etale ringed topoi.  Apply the functoriality of the geometric Sen
morphism
\cite[Theorem~3.3.4, diagram~(3.12)]{RodriguezCamargo2026} to the preceding
$G$-equivariant diagram.  All terms in the following square are sheaves on
$X_{\proet}$:
\[
\begin{CD}
  f_{\proet}^*\mathfrak g_{\pi_Y}^\vee
    \otimes_{\mathbb Q_p}\Oh_X
    @>{f_{\proet}^*\theta_{\widetilde Y}}>>
  f^*\Omega^1_{Y/\Cp}\otimes_{\OO_X}\Oh_X(-1) \\
  @V{\simeq}VV
    @VV{df\otimes 1}V \\
  \mathfrak g_{\pi_X}^\vee\otimes_{\mathbb Q_p}\Oh_X
    @>{\theta_{\widetilde X}}>>
  \Omega^1_{X/\Cp}\otimes_{\OO_X}\Oh_X(-1).
\end{CD}
\]
The left vertical isomorphism is induced by
$\operatorname{id}_{\Lie(G)}$.  Thus
\[
  \theta_{\widetilde X}
    =(df\otimes1)\circ f_{\proet}^*\theta_{\widetilde Y}.
\]
Combining this square with \eqref{eq:image-df} gives
\begin{equation}\label{eq:image-theta}
  \im\bigl(\theta_{\widetilde X}\bigr)
  \subseteq t\Oh_X\,dt(-1)
\end{equation}
as subsheaves of
$\Omega^1_{X/\Cp}\otimes_{\OO_X}\Oh_X(-1)$ on $X_{\proet}$.

By the construction,
\[
  \widetilde X=\Spa(B_\infty,B_\infty^\circ)
\]
is affinoid perfectoid.  By $v$-descent for vector bundles and affinoid
$v$-acyclicity
\cite[Theorem~17.1.3 and Lemma~17.1.8]{ScholzeWeinstein2020}, it is enough
to show that the induced map on sections
\[
  \Gamma(\widetilde X,\theta_{\widetilde X})\colon
  B_\infty\otimes_{\mathbb Q_p}\Lie(G)^\vee
  \longrightarrow B_\infty\,dt(-1)
\]
is not surjective.  Equation~\eqref{eq:image-theta} gives
\[
  \im\bigl(\Gamma(\widetilde X,\theta_{\widetilde X})\bigr)
  \subseteq tB_\infty\,dt(-1).
\]
The Zariski closed subset
\[
  V_{\widetilde X}(t)=\widetilde X\times_XV_X(t)
\]
is nonempty because $\pi_X$ is a torsor.  Zariski closed and strongly
Zariski closed subsets of an affinoid perfectoid space coincide
\cite[Theorem~7.4 and Remark~7.5]{BhattScholze2022}.  Hence
$V_{\widetilde X}(t)$ is represented by a nonzero perfectoid Tate quotient
$B_\infty\twoheadrightarrow C$ on which $t$ vanishes.  Consequently
\[
  C\otimes_{B_\infty}
  \frac{B_\infty\,dt(-1)}{tB_\infty\,dt(-1)}
  =C\,dt(-1)\neq0.
\]
Thus $tB_\infty\,dt(-1)$ is a proper submodule of
$B_\infty\,dt(-1)$, so the displayed map on sections is not surjective.
By $v$-descent, $\theta_{\widetilde X}$ is not surjective on
$X_{\proet}$.
\end{proof}

\subsection{Finite pullback and diamondian perfectoid towers}
\label{sec:diamondian-pullback}
The goal of this subsection is to compare our construction and the one in \cite[Remark 5.5]{BCH25}.
Let $K/\mathbb Q_p$ be a complete nonarchimedean field. The key of construction in \cite[Remark 5.5]{BCH25} is that finite
pullback preserves diamondian affinoid perfectoidness.

\begin{proposition}
\label{prop:finite-diamondian-pullback}
Let $X$ and $Y$ be affinoid rigid analytic spaces over $K$, let
$f\colon X\to Y$ be a finite morphism, and let
$\mathcal Y=(Y_\lambda)_{\lambda\in\Lambda}$ be a diamondian affinoid
perfectoid profinite \'etale tower over $Y$.  Then
\[
  \mathcal X=(X\times_YY_\lambda)_{\lambda\in\Lambda}
\]
is diamondian affinoid perfectoid.
\end{proposition}

\begin{proof}
Put $X_\lambda=X\times_YY_\lambda$, and let
$P=\Spa(R,R^+)$ be an affinoid perfectoid space representing
$\mathcal Y^\diamond$.  Since inverse limits of v-sheaves commute with
fibre products,
\begin{equation}\label{eq:finite-diamond-pullback}
  \varprojlim_{\lambda\in\Lambda}X_\lambda^\diamond
   \simeq
  X^\diamond\times_{Y^\diamond}P^\diamond.
\end{equation}
Evaluation of $Y^\diamond$ on $P$ identifies the structural morphism
$P^\diamond\to Y^\diamond$ with a morphism $P\to Y$.  Put
$Q=X\times_YP$.  Then $Q\to P$ is finite, hence
$Q=\Spa(S,S^+)$ is affinoid and
$(R,R^+)\to(S,S^+)$ is a finite morphism of analytic Huber pairs; in
particular, $R^+\to S^+$ is integral.  By
\cite[Proposition~8.5 and Theorem~10.11]{BhattScholze2022}, the
perfectoidization $(S^+)^{\mathrm{perfd}}$ is a classical, integrally
perfectoid ring and satisfies
\[
  Q^\diamond=\operatorname{Spd}(S,S^+)
   \simeq
  \Spa\bigl((S^+)^{\mathrm{perfd}}[1/p],
             (S^+)^{\mathrm{perfd}}\bigr)^\diamond.
\]
Therefore \eqref{eq:finite-diamond-pullback} is represented by the
affinoid perfectoid space in the display.
\end{proof}

\begin{remark}\label{rem:bchh-example}
In \cite[Remark~5.5]{BCH25}, Bellovin--Cai--Howe
construct a diamondian affinoid perfectoid tower which is not affinoid
perfectoid. We revisit their example and spell out some details of the
argument. Let $C=\mathbb C_p$, choose
$\epsilon\in|C^\times|$ with $0<\epsilon<1$, and set
\[
  R=C\langle\epsilon^{-2}s\rangle,
  \qquad
  Y=\mathbb D_{\epsilon^2}=\Spa(R,R^\circ).
\]
For $n\geq0$, put
\[
  B_n=R[u_n]/\bigl(u_n^{p^n}-(1+s)\bigr),
  \qquad
  Y_n=\Spa(B_n,B_n^\circ),
\]
with transition homomorphism $B_n\to B_{n+1}$ defined by
$u_n\mapsto u_{n+1}^p$.  Since $|s|<1$, the element $1+s$ is a unit, so
each $Y_n\to Y$ and each transition map is finite \'etale.  If
\[
  T=C\langle u^{\pm1/p^\infty}\rangle
    =\widehat{\varinjlim_n
       C\langle u^{\pm1/p^n}\rangle}^{\,\rho},
\]
then
\[
  \widehat{\varinjlim_nB_n}^{\,\rho}
    =T\langle\epsilon^{-2}(u-1)\rangle,
  \qquad u=1+s.
\]
This is a rational localization of the perfectoid torus $\Spa(T,T^\circ)$
and is affinoid perfectoid
\cite[Theorem~6.3]{Scholze2012}; hence
$\mathcal Y=(Y_n)_{n\geq0}$ is affinoid perfectoid.

Put
\[
  A=C\langle\epsilon^{-1}t\rangle,
  \qquad
  X=\mathbb D_\epsilon=\Spa(A,A^\circ),
\]
and define $g\colon X\to Y$ by $g^\sharp(s)=t^2$.  The $R$-module
$A=R\oplus tR$ is finite free.  At level $n$, the pulled-back tower has
affinoid algebra
\begin{equation}\label{eq:bchh-pullback-algebra}
  A_n=A\widehat\otimes_RB_n
      =B_n[t]/(t^2-s)=B_n\oplus tB_n.
\end{equation}
Proposition~\ref{prop:finite-diamondian-pullback} shows that
$(\Spa(A_n,A_n^\circ))_{n\geq0}$ is diamondian affinoid perfectoid.

Let $y_n\in\mathcal M(B_n)$ be the point $s=0$, $u_n=1$, and let
$x_n\in\mathcal M(A_n)$ be its lift with $t=0$.  There is a
$C$-derivation
\[
  \delta_n\colon A_n\longrightarrow C,
  \qquad
  \delta_n(t)=1,
  \qquad
  \delta_n|_{B_n}=0,
\]
where $C$ is an $A_n$-module through evaluation at $x_n$.  Equivalently,
differentiating $u_n^{p^n}=1+t^2$ at $x_n$ gives
\[
  p^n u_n(x_n)^{p^n-1}\delta_n(u_n)
    =2t(x_n)\delta_n(t)=0,
  \qquad\text{hence}\qquad \delta_n(u_n)=0.
\]
The identities $u_n=u_{n+1}^p$ make the $\delta_n$ compatible.

We verify continuity for the spectral norms.  Write
$\rho_{B_n}$ and $\rho_{A_n}$ for the spectral norms.  If $\eta$ is the
Gauss point of $\mathcal M(R)$ and the characteristic polynomial of
$b\in B_n$ over $R$ is
\[
  T^d+c_1T^{d-1}+\cdots+c_d,
\]
then the nonarchimedean root bound and the Gauss norm give
\begin{equation}\label{eq:bchh-boundary}
  \rho_{B_n}(b)
   =\sup_{z\in\mathcal M(R)}\max_{1\leq i\leq d}|c_i(z)|^{1/i}
   =\max_{1\leq i\leq d}|c_i(\eta)|^{1/i}
   =\max_{y\mapsto\eta}|b(y)|.
\end{equation}
For $b\in B_n$, every point $x\in\mathcal M(A_n)$ above
$y\in\mathcal M(B_n)$ satisfies $|t(x)|=|s(y)|^{1/2}\leq\epsilon$.
Conversely, a point $y$ realizing the last maximum in
\eqref{eq:bchh-boundary} lifts to $\mathcal M(A_n)$ and satisfies
$|s(y)|=\epsilon^2$.  Therefore
\begin{equation}\label{eq:bchh-t-norm}
  \rho_{A_n}(tb)=\epsilon\rho_{B_n}(b).
\end{equation}
For $f=b_0+tb_1$ as in \eqref{eq:bchh-pullback-algebra}, let
$\sigma$ be the involution fixing $B_n$ and sending $t$ to $-t$.
Since $\sigma$ preserves the spectral norm,
\[
  tb_1=\frac{f-\sigma(f)}2,
  \qquad
  \rho_{A_n}(tb_1)\leq |2|^{-1}\rho_{A_n}(f).
\]
Equations \eqref{eq:bchh-t-norm} and
$\delta_n(f)=b_1(y_n)$ now give the uniform estimate
\begin{equation}\label{eq:bchh-derivation-bound}
  |\delta_n(f)|
   \leq\rho_{B_n}(b_1)
   =\epsilon^{-1}\rho_{A_n}(tb_1)
   \leq (|2|\epsilon)^{-1}\rho_{A_n}(f).
\end{equation}
Thus the compatible derivations extend to a continuous $C$-derivation
\[
  \delta\colon
  \widehat{\varinjlim_nA_n}^{\,\rho}\longrightarrow C.
\]
The compatible evaluations at $x_n$ extend to an evaluation
$\chi\colon\widehat{\varinjlim_nA_n}^{\,\rho}\to C$, relative to which
$\delta$ is a derivation, and $\delta(t)=1$.

Suppose that the pulled-back tower were affinoid perfectoid, and put
$P=\widehat{\varinjlim_nA_n}^{\,\rho}$.  Then $P$ would be a perfectoid
$C$-algebra by Lemma~\ref{lem:integral-uniform-comparison}.  Continuity of
$\delta$ makes
\[
  M=\sup_{a\in P^\circ}|\delta(a)|
\]
finite.  For $a\in P^\circ$, surjectivity of Frobenius on $P^\circ/p$
gives $a=b^p+pc$ with $b,c\in P^\circ$.  Since
$|\chi(b)|\leq1$,
\[
  |\delta(a)|
   =\bigl|p\chi(b)^{p-1}\delta(b)+p\delta(c)\bigr|
   \leq |p|M.
\]
Taking the supremum over $a\in P^\circ$ gives $M\leq|p|M$, hence $M=0$,
contrary to $\delta(t)=1$.  Thus the pulled-back tower is not affinoid
perfectoid; this is the derivation argument of
\cite[Lemma~3.9 and Remark~5.5]{BCH25}.
\end{remark}

\begin{remark}
    The example in Remark~\ref{rem:bchh-example} disproves the geometric
Ax--Sen--Tate density statement: for a profinite \'etale cover
$U\to X$, the ring $\OO_X(U)$ need not be dense in $\Oh_X(U)$
for the topology with neighborhood basis $p^m\Oh_X^+(U)$,
$m\geq0$. This answers negatively the question raised immediately
after \cite[Lemma~4.2]{Scholze2013}, already for a $\Zp$-torsor
over a smooth affinoid curve.
\end{remark}

\section{Generic surjectivity for curves}\label{sec:generic-curves}

The goal of this section is to prove Theorem~\ref{thm:intro-curve}. 
Let $X/\Cp$ be a smooth connected rigid analytic curve and
$\mathcal X=(X_J)_J$ a profinite \'etale $G$-tower, where $G$ is a
compact $p$-adic Lie group and $J$ ranges over its normal open subgroups.
Assume $\mathcal X$ is perfectoid in the sense of
Definition~\ref{def:four-perfectoid-notions}\textup{(3)}: after passing to a finite level it has an
affinoid open covering on which the tail $X_J=\Spa(A_J,A_J^\circ)$ gives
an affinoid perfectoid pair
\begin{equation}\label{eq:perfectoid-completion}
 A_\infty^+=\varprojlim_m
 \bigl(\varinjlim_J A_J^\circ\bigr)/p^m,
 \qquad A_\infty=A_\infty^+[1/p].
\end{equation}
Write $\pi:\widetilde X\to X$ for the resulting perfectoid space and
$\mathfrak g_\pi$ for the adjoint local system attached to
$\mathfrak g=\Lie G$. Its geometric Sen map is
\[
 \theta_{\widetilde X}:\Oh_X\otimes_{\Qp}\mathfrak g_\pi^\vee
 \longrightarrow\Oh_X\otimes_{\OO_X}\Omega^1_{X/\Cp}(-1).
\]

We will prove the following theorem.

\begin{theorem}\label{thm:curve}
The geometric Sen map is surjective at every type $2$ point of $X$.
Its nonsurjectivity locus is a locally finite closed analytic subset,
consisting of classical points, and finite when $X$ is quasicompact.
Its complement is Zariski open and dense.
\end{theorem}

\subsection{Reduction to type 2 points}\label{sec:reduction}

\begin{lemma}\label{lem:branch-completion}
A compatible branch $(x_J)_J$ above a rank-one point $x\in X$ determines
$x_\infty\in\widetilde X$ with
\[
 \mathcal H(x_\infty)=\widehat L,
 \qquad L=\bigcup_J\mathcal H(x_J).
\]
This field is perfectoid. The perfectoid tower property is preserved
by \'etale localization at a finite level.
\end{lemma}

\begin{proof}
Choose a finite-level affinoid neighborhood as in
\eqref{eq:perfectoid-completion}. By
Lemma~\ref{lem:integral-uniform-comparison}, $A_\infty$ is the
spectral completion of $\varinjlim_J A_J$.
The branch seminorms, bounded by the spectral norm, extend to
$x_\infty$. Each $\operatorname{Frac}(A_J/\ker x_J)$ is dense in
$\mathcal H(x_J)$; density in $A_\infty$ and continuity of division
show that their union is dense in $\mathcal H(x_\infty)$.
This gives the equality, and \cite[Corollary 6.7(ii)]{Scholze2012}
gives perfectoidness. The localization assertion is
\cite[Lemma 4.5(i)]{Scholze2013}.
\end{proof}

\begin{lemma}\label{lem:faithful-differential}
There is a continuous $\rho:G\to\GL_r(\Qp)$ with injective differential.
\end{lemma}

\begin{proof}
Integrate a faithful representation of $\mathfrak g$ supplied by Ado's
theorem on a sufficiently small open subgroup $G_0$. Its finite-dimensional induction to $G$ restricts
to a representation containing the original one, so has faithful differential.
\end{proof}

\begin{lemma}\label{lem:analytic-rank-locus}
The nonsurjectivity locus of $\theta_{\widetilde X}$ is closed analytic.
If it contains a type $2$ point $x$, then $\theta_{\widetilde X}$ vanishes
on a neighborhood of $x$.
\end{lemma}

\begin{proof}
\'Etale locally on $X$, we can trivialize a stable adjoint lattice of $\mathfrak g_\pi$ modulo $p^m$ for  $m\gg0$ and take connected affinoid $W$ with a small toric
chart (a composite of rational localizations and finite \'etale maps).
Write $W_\infty^{\mathrm{tor}}\to W$ for its toric $\Gamma\simeq\Zp$-tower,
with finite-level rings $A_n$ and completion $A_\infty^{\mathrm{tor}}$.
For $P=(\Oh_X\otimes_{\Qp}\mathfrak g_\pi)(W_\infty^{\mathrm{tor}})$,
relative Sen descent gives finite projective $A_n$-modules $P_n$ with
\[
 A_\infty^{\mathrm{tor}}\otimes_{A_n}P_n\simeq P,
 \qquad P^{\Gamma\text{-}\mathrm{la}}=\varinjlim_nP_n
\]
by \cite[Proposition 2.2.14 and Theorem 2.4.4(1),(3)]{RodriguezCamargo2026}.
For a generator $e$ of $\Omega^{1,\vee}_{W/\Cp}(1)$, the
$\Gamma$-invariant section $s=\theta_{\widetilde X}^\vee(e)$ belongs to
some $P_n$. Its zero locus $Z_n\subset W_n$ detects nonsurjectivity,
since $\Omega^1_W$ has rank one. 
By $\Gamma$-invariance, the bad locus is therefore the finite image of $Z_n$, hence closed
analytic; the same argument descends the initial cover, and these loci glue.

A proper closed analytic subset of a smooth connected affinoid curve
contains only classical points. A type $2$ bad point therefore has an
entire connected neighborhood in the bad locus. On its finite-level
preimage $s$ vanishes fiberwise, hence vanishes since $W_n$ is reduced.
Thus $\theta_{\widetilde X}=0$ there.
\end{proof}

\begin{lemma}\label{lem:zero-sen-period}
Suppose $\theta_{\widetilde X}=0$ on a nonempty analytic open $W$, and let
$\rho:G\to\GL_r(\Qp)$ be continuous. After \'etale localization and
shrinking, every type $2$ point $x\in W$ and branch above it admits
\begin{equation}\label{eq:geometric-period}
 M\in\GL_r(\widehat L),\qquad
 \sigma(M)=M\rho(\sigma)\quad(\sigma\in D),
\end{equation}
where $K=\mathcal H(x)$, $L$ is the branch extension and
$D=\Gal(L/K)\subseteq G$.
\end{lemma}

\begin{proof}
The associated local system $V_\pi$ has zero Sen map.
After \'etale localization and
shrinking, zero-Sen descent
\cite[Propositions 3.2.3(3) and 3.2.4]{RodriguezCamargo2026} gives
an analytic vector bundle $\mathcal E$ with
\[
 \Oh_W\otimes_{\OO_W}\mathcal E
 \simeq\Oh_W\otimes_{\Qp}V_\pi.
\]
At $C_x=\widehat{\overline K}$ a $K$-basis of $\mathcal E_x$ gives a
matrix $P$ satisfying $\rho(\sigma)\sigma(P)=P$. Its inverse $M$
satisfies \eqref{eq:geometric-period}; its entries are fixed by
$\ker(G_K\to D)$ and hence lie in $\widehat L$ by Ax--Sen--Tate
\cite[Proposition 3.8]{FontaineOuyang}.
\end{proof}

\begin{proposition}\label{prop:type2-reduction}
Failure of Theorem~\ref{thm:curve} at a type $2$ point gives a type $2$
field $K/\Cp$, a Galois extension $L/K$ with compact $p$-adic Lie group
$D$, and $\rho:D\to\GL_r(\Qp)$ with injective differential such that
$\widehat L$ is perfectoid and \eqref{eq:geometric-period} holds.
\end{proposition}

\begin{proof}
Apply Lemmas~\ref{lem:analytic-rank-locus}, \ref{lem:faithful-differential}
and \ref{lem:zero-sen-period}. Finite \'etale extension preserves type $2$,
$d\rho|_{\Lie D}$ remains injective, and
Lemma~\ref{lem:branch-completion} gives perfectoidness of $\widehat L$.
\end{proof}

\subsection{Wewers's ramification theorem}\label{sec:wewers}

Let $F$ be a complete discretely valued field of mixed
characteristic $(0,p)$, normalized by $v(p)=1$. Put
$\kappa=\kappa(F)$ and assume
\begin{equation}\label{eq:wewers-residue-hypotheses}
 [\kappa:\kappa^p]=p,\qquad
 H^1(\kappa,\mathbf F_p)\ne0.
\end{equation}
Let $k$ be the relative algebraic closure of
$\operatorname{Frac}W(\bigcap_{j\ge0}\kappa^{p^j})$ in $F$, called the constant field of $F$.
Assume also that $F/k$ is weakly unramified, meaning that
$v(k^\times)=v(F^\times)$, and that $\mu_p\subset F$.
These are the hypotheses of \cite[\S2]{Wewers2014}.

\begin{definition}\label{def:fierce}
A finite extension $E/F$ of complete valued fields is fierce if
$\kappa(E)/\kappa(F)$ is purely inseparable and
$[E:F]=[\kappa(E):\kappa(F)]$.
\end{definition}

For a finite fierce Galois extension $E/F$ with Galois group $Q$ and $\sigma\in Q$, set
\begin{equation}\label{eq:ramification-numbers}
 i_Q(\sigma)=\inf_{z\in E^\times}
       \bigl(v(\sigma z-z)-v(z)\bigr),\qquad i_Q(1)=+\infty.
\end{equation}
Since $v(E^\times)=v(F^\times)$, this equals
$\min_{z\in\OO_E}v(\sigma z-z)$, the definition in
\cite[\S3.1]{Wewers2014}. Write
\begin{align}
 Q_t&=\{\sigma:i_Q(\sigma)\ge t\},&
 \Phi_Q(t)&=\int_0^t|Q_s|\,ds
           =\sum_{\sigma\in Q}\min\{i_Q(\sigma),t\},
 \label{eq:herbrand}\\
 \Psi_Q&=\Phi_Q^{-1},& Q[a]&=Q_{\Psi_Q(a)}.
 \label{eq:upper-filtration}
\end{align}
For $Q\ne1$ put
\[
 u(Q)=\max_{\sigma\ne1}i_Q(\sigma),\qquad
 b(Q)=\Phi_Q(u(Q)),
\]
and put $u(1)=b(1)=0$. Then
\begin{equation}\label{eq:herbrand-tail}
 u(Q)=\int_0^{b(Q)}\frac{da}{|Q[a]|},\qquad
 \Psi_Q(a)=u(Q)+a-b(Q)\quad(a\ge b(Q)).
\end{equation}
The group is retained at its break. For a nonzero character
$\eta:Q\to\Qp/\Zp$, define
\begin{equation}\label{eq:character-conductor}
 c(\eta)=\max\{a:\eta|_{Q[a]}\ne0\}=b(Q/\ker\eta),
 \qquad c(0)=0.
\end{equation}
The last equality follows from compatibility of upper numbering
with quotients \cite[Remark 3.1 and Definition 3.3]{Wewers2014}.
The notation in \eqref{eq:ramification-numbers}--
\eqref{eq:character-conductor} will also be used for the extensions
of complete rank-one fields below.

\begin{lemma}\label{lem:cyclic-breaks}
Suppose $A=\Gal(E/F)\simeq\Z/p^n\Z$, and choose a faithful
character $\chi:A\hookrightarrow\Qp/\Zp$. In additive notation set
\[
 \chi_i=p^{n-i}\chi,\qquad E_i=E^{p^iA},\qquad
 \delta_i=c(\chi_i)\quad(1\le i\le n),\qquad\delta_0=0.
\]
Then $[E_i:F]=p^i$ and
\begin{equation}\label{eq:delta-quotient}
 \delta_i=b(E_i/F)=b(A/p^iA),\qquad
 0<\delta_1<\cdots<\delta_n.
\end{equation}
Moreover,
\begin{equation}\label{eq:cyclic-intervals}
 A[a]=p^{i-1}A\quad(\delta_{i-1}<a\le\delta_i),
 \qquad A[a]=0\quad(a>\delta_n).
\end{equation}
\end{lemma}

\begin{proof}
The kernel of $\chi_i$ is $p^iA$, which gives the degree and the
first equality by \eqref{eq:character-conductor}. The strict
inequalities and $|A[\delta_i]|=p^{n-i+1}$ are
\cite[Corollary 3.6 and Remark 4.2]{Wewers2014}.
Uniqueness of subgroups of each order in a cyclic group gives
\eqref{eq:cyclic-intervals}.
\end{proof}

Put $h=1/(p-1)$ and $H=h+1=p/(p-1)$.

\begin{theorem}[Wewers]\label{thm:wewers}
Under the hypotheses above, the breaks in
Lemma~\ref{lem:cyclic-breaks} satisfy $0<\delta_1\le H$ and, for
$i>1$,
\begin{equation}\label{eq:wewers-breaks}
 \begin{cases}
 p\delta_{i-1}\le\delta_i\le H,&\delta_{i-1}\le h,\\
 \delta_i=\delta_{i-1}+1,&\delta_{i-1}>h.
 \end{cases}
\end{equation}
\end{theorem}

\begin{proof}
This is \cite[Theorem 4.3(i)--(iii)]{Wewers2014}, with
$\delta_i=\operatorname{sw}(p^{n-i}\chi)$ as in
\cite[Definition 4.1]{Wewers2014}.
\end{proof}

\subsection{Unramified Lie bases and approximation}\label{sec:approximation}

\begin{definition}\label{def:classC}
Let $\mathcal C$ denote the class of complete valued fields $F$ for which there exists a completed residue field $K_0$ of a type~2 point on a smooth $\Cp$-curve, and an unramified Galois extension
$U/K_0$ with compact $p$-adic Lie Galois group, such that $F=\widehat U$.
\end{definition}

\begin{lemma}\label{lem:AS-Lie}
For $F=\widehat U\in\mathcal C$,
\begin{equation}\label{eq:lie-base-wewers}
 [\kappa(F):\kappa(F)^p]=p,\qquad
 \dim_{\mathbf F_p}H^1(\kappa(F),\mathbf F_p)=\infty.
\end{equation}
In particular these assertions hold for
$U=L^I$, where $L/K_0$ is Galois with compact $p$-adic Lie group
$D$ and inertia subgroup $I$.
\end{lemma}

\begin{proof}
Put $k_0=\kappa(K_0)$ and $k=\kappa(F)=\kappa(U)$.
The extension $k/k_0$ is Galois with group
$\Gamma=\Gal(U/K_0)$. A $p$-basis of the one-variable function
field $k_0/\overline{\mathbf F}_p$ remains a $p$-basis under
separable algebraic extension, proving the first assertion.

Choose a discrete valuation $w$ of $k_0$ and a uniformizer $t$.
The classes $[t^{-m}]$ with $m>0$ and $p\nmid m$ are independent
in $k_0/\wp(k_0)$, where $\wp(b)=b^p-b$: a nonzero combination
has pole order prime to $p$, whereas a pole of $\wp(b)$ has order
divisible by $p$. Inflation--restriction gives
\[
 0\longrightarrow\Hom_{\mathrm{cont}}(\Gamma,\mathbf F_p)
 \longrightarrow H^1(k_0,\mathbf F_p)
 \longrightarrow H^1(k,\mathbf F_p)^\Gamma.
\]
A compact $p$-adic Lie group is topologically finitely generated,
so the leftmost space is finite dimensional. The last map has
infinite-dimensional image, proving the second assertion.

For the last assertion, since $I$ is the inertia subgroup, $U/K_0$ is unramified with Galois group 
$D/I$, a compact $p$-adic Lie group, thus
$F=\widehat U\in\mathcal C$.
\end{proof}

\begin{lemma}\label{lem:henselian-approximation}
Let $V$ be a henselian rank-one field of characteristic zero.
Every finite Galois extension $E/\widehat V$ of degree $q$ descends
to a Galois extension $V(z_0)/V$ of degree $q$, with $z_0$
arbitrarily close to a prescribed primitive element $z\in E$.
\end{lemma}

\begin{proof}
Approximate the minimal polynomial of $z$ over $\widehat V$
by a monic degree-$q$ polynomial over $V$. Hensel's and Krasner's
lemmas give a nearby root $z_0\in E$ with
$\widehat V(z_0)=E$, hence $[V(z_0):V]=q$ and
$\widehat{V(z_0)}=E$. The field $V(z_0)$ is henselian and
therefore separably closed in its completion
\cite[Theorem~1.9.5(ii)]{TemkinNotes}. Every conjugate of $z_0$
lies in $E$ and hence already in $V(z_0)$, proving Galois descent.
\end{proof}

\begin{lemma}\label{lem:unramified-approximation}
Let $F=\widehat U\in\mathcal C$ and let $E/F$ be finite Galois
with group $Q$.
There are complete discretely valued subfields
$K_{\mathrm d}\subset F$ and $E_{\mathrm d}\subset E$ such that
\[
 \kappa(K_{\mathrm d})=\kappa(F)=\kappa(U),\qquad
 \kappa(E_{\mathrm d})=\kappa(E),\qquad
 E_{\mathrm d}\otimes_{K_{\mathrm d}}F\simeq E.
\]
The extension $E_{\mathrm d}/K_{\mathrm d}$ is Galois, and
restriction identifies its Galois group and ramification numbers
with those of $E/F$, also for corresponding intermediate extensions.
The field $K_{\mathrm d}$ satisfies the base-field hypotheses of
Theorem~\ref{thm:wewers}. Moreover,
\[
 [E:F]=[\kappa(E):\kappa(F)],
\]
and $\mathcal C$ is closed under finite extensions.
\end{lemma}

\begin{proof}
Put
$q=[E:F]$ and $R=\kappa(U)=\kappa(F)$. The field $K_0$ is
stable \cite[Example~4.2.5(i)]{Temkin2017}, so its unramified
extension $U$ is defectless \cite[Proposition~2.18]{Kuhlmann2010}.
Lemma~\ref{lem:henselian-approximation} shows that completion
preserves defectlessness here: finite Galois extensions descend
with their degrees, and completion changes neither residue fields
nor value groups. Thus $F$ is defectless. Since its value group
is divisible, $q=[\kappa(E):R]$. Choose units
$b_1,\ldots,b_q\in E$ lifting an $R$-basis of $\kappa(E)$.

We now construct a directed family of complete discretely valued
subfields of $F$, all with residue field $R$, whose union is dense.
Put $B=\operatorname{Frac}W(\overline{\mathbf F}_p)\subset\Cp$.
Choose a lift $T\in\OO_{K_0}$ of a separating residue parameter
$t$ of $\kappa(K_0)/\overline{\mathbf F}_p$.
Then $K_0/K_{\mathrm G}$, where
$K_{\mathrm G}=\widehat{\Cp(T)}$, is finite unramified by
\cite[Corollary~6.3.4 and Theorem~6.3.1(iii)]{Temkin2010}.
The field $D=\widehat{B(T)}$ is complete discretely valued with
uniformizer $p$ and residue field $\overline{\mathbf F}_p(t)$.
Since $R/\overline{\mathbf F}_p(t)$ is separable algebraic, the
unramified correspondence gives an unramified algebraic extension
$D_R/D$ inside $U$ with residue field $R$. Set
$A=\widehat{D_R}\subset F$; it is still complete discretely
valued with uniformizer $p$ and residue field $R$.
For each finite extension $k/B$ inside $\Cp$ containing $\mu_p$,
put $K_k=Ak$. Since $k/B$ is totally ramified, an Eisenstein
polynomial for its uniformizer remains Eisenstein over $A$.
Thus $K_k$ is complete discretely valued and
\[
 \kappa(K_k)=\kappa(U)=R,\qquad v(K_k^\times)=v(k^\times).
\]
The union $V=\varinjlim_k K_k$ is henselian and dense in $F$:
its closure contains $\Cp$, $D_R$, and hence $K_{\mathrm G}$;
the unramified correspondence gives $D_RK_{\mathrm G}=U$.
Each $K_k$ has the residue-field properties of
Lemma~\ref{lem:AS-Lie}. Its constant field contains $k$, so the
equality of value groups makes $K_k$ weakly unramified. Together
with $\mu_p\subset k$, these are the base-field hypotheses of
Theorem~\ref{thm:wewers}.

Apply Lemma~\ref{lem:henselian-approximation} to $V\subset F$.
It gives $w\in E$ such that $V(w)/V$ is Galois of degree $q$ and
$F(w)=E$. Since $V(w)$ is dense in $E$, we can  choose
$\beta_j\in V(w)$ with $|\beta_j-b_j|<1$ for $1\le j\le q$.
The defining polynomial of $w$ and the power-basis expressions
of all $\sigma(w)$ and $\beta_j$ involve finitely many
coefficients in $V$, all of which lie in one $K_k$.
Put $K_{\mathrm d}=K_k$ and $E_{\mathrm d}=K_k(w)$.
Then $E_{\mathrm d}/K_{\mathrm d}$ is Galois of degree $q$,
with group $Q$ identified by restriction, and
$E_{\mathrm d}\otimes_{K_{\mathrm d}}F\simeq E$.
Since $\bar\beta_j=\bar b_j$, its residue field contains an
$R$-basis of $\kappa(E)$, proving
$\kappa(E_{\mathrm d})=\kappa(E)$.
The common orthonormal basis $\beta_1,\ldots,\beta_q$ gives
$\|\sigma-1\|=\max_j|\sigma\beta_j-\beta_j|$ in both fields,
and hence
\begin{equation}\label{eq:discrete-exact-numbers}
 i_{E_{\mathrm d}/K_{\mathrm d}}(\sigma)=i_{E/F}(\sigma).
\end{equation}
The degree equality and the fundamental inequality in towers
likewise give $\kappa(E_{\mathrm d}^H)=\kappa(E^H)$ for every
subgroup $H\subseteq Q$. Common residue bases in these fields
then prove the same assertion for all intermediate extensions.

To prove closure under finite extensions, put $S=\kappa(E)$.
Choose a finite extension $S_0/\kappa(K_0)$ inside $S$ with
$S=RS_0$. Then $S/S_0$ is Galois with compact $p$-adic Lie group,
a open subgroup of $\Gal(R/\kappa(K_0))$.
Choose a separating parameter $s$ of
$S_0/\overline{\mathbf F}_p$ and a unit lift $T'\in E_{\mathrm d}$.
The field $H_0=\widehat{k(T')}$ and $E_{\mathrm d}$ have a common
uniformizer. The unramified algebraic extension $H_S/H_0$ inside
$E_{\mathrm d}$ with residue field $S$ is therefore dense in
$E_{\mathrm d}$, by successive approximation with this uniformizer.
Set $G_0=\widehat{\Cp(T')}$ and $U'=H_SG_0\subset E$.
The subextension of $U'/G_0$ corresponding to $S_0$ is a type~2
field $K_0'$, and $U'/K_0'$ is unramified Galois with compact
$p$-adic Lie group. Its completion contains $E_{\mathrm d}$ and
$\Cp$, hence $F=\widehat{A\Cp}$ and $E=E_{\mathrm d}F$.
Thus $E=\widehat{U'}\in\mathcal C$. The same construction applies
to the discrete models of intermediate fields of a Galois closure,
proving closure under arbitrary finite extensions.
\end{proof}

\begin{lemma}\label{lem:inertia-reduction}
Let $K_0$ be type~2 over $\Cp$, let $U/K_0$ be unramified
Galois with compact $p$-adic Lie group, and put $F=\widehat U$.
For a finite Galois extension $E/F$ with group $Q$ and inertia
$I$, put $F'=E^I$. Then $F'/F$ is unramified,
$F'\in\mathcal C$, $E/F'$ is fierce, and $\kappa(F')$ satisfies
\eqref{eq:lie-base-wewers}. Moreover,
\begin{equation}\label{eq:inertia-positive}
 \Phi_{E/F}=\Phi_{E/F'},\qquad Q[a]=I[a]\quad(a>0).
\end{equation}
\end{lemma}

\begin{proof}
The unramified correspondence identifies $F'/F$ with the maximal
separable residue subextension. Lemma~\ref{lem:unramified-approximation}
gives $F'\in\mathcal C$ and
$[E:F']=[\kappa(E):\kappa(F')]$, so $E/F'$ is fierce.
Apply Lemma~\ref{lem:AS-Lie} to a presentation
$F'=\widehat{U'}$ with $U'/K_0'$ unramified Galois and Lie group.
If $\sigma\notin I$, it moves a residue class, giving $i_Q(\sigma)=0$.
For $\sigma\in I$, the norm of $\sigma-1$ on $E$ is unchanged.
Thus $Q_t=I_t$ for $t>0$, which proves \eqref{eq:inertia-positive}.
\end{proof}

\begin{proposition}\label{prop:abelian-power-filtration}
Let $K_0$ be type~2 over $\Cp$, let $U/K_0$ be unramified
Galois with compact $p$-adic Lie group, and put $F=\widehat U$.
For every finite abelian $p$-extension $E/F$ with group $A$,
\begin{align}
 (A[a])^p&\subseteq A[pa]&&(0<a\le h),\label{eq:A3-low}\\
 (A[a])^p&=A[a+1]&&(a>h).\label{eq:A3-high}
\end{align}
\end{proposition}

\begin{proof}
Let $I$ be the inertia group and put $F'=E^I$.
Lemma~\ref{lem:inertia-reduction} gives $F'\in\mathcal C$,
the residue conditions \eqref{eq:lie-base-wewers}, and
$A[a]=I[a]$ for $a>0$. The case $I=1$ is immediate.
Otherwise apply Lemma~\ref{lem:unramified-approximation} to $E/F'$.
The discrete model has the same residue fields, so it is fierce;
it has the same filtered group $I$ and satisfies Wewers's hypotheses. Quotient compatibility
(Lemma~\ref{lem:classC-Herbrand}) therefore transfers
Theorem~\ref{thm:wewers} to its cyclic characters. For
$\chi\in I^\vee=\Hom(I,\Qp/\Zp)$, put $x=c(p\chi)$,
$y=c(\chi)$. The resulting inequalities are
\begin{equation}\label{eq:classC-conductors}
 px\le y\le H\quad(x\le h),\qquad y=x+1\quad(x>h),
\end{equation}
with the degree-$p$ bound including the case  $x=0$.

Write $I$ additively. The annihilators are
\[
 (I[a])^\perp=\{\chi:y<a\},\qquad
 (pI[a])^\perp=\{\chi:x<a\}.
\]
For $a>h$, \eqref{eq:classC-conductors} gives
$y<a+1\Longleftrightarrow x<a$, hence $pI[a]=I[a+1]$.
For $0<a\le h$, it gives $x\ge a\Longrightarrow y\ge pa$,
hence $(I[pa])^\perp\subseteq(pI[a])^\perp$ and
$pI[a]\subseteq I[pa]$. Substitute $I[a]=A[a]$ completes the proof.
\end{proof}

\begin{lemma}\label{lem:classC-Herbrand}
For $E/F$ finite Galois with $F\in\mathcal C$, group $Q$,
and $M=E^N$ for $N\triangleleft Q$,
\[
 \Phi_{E/F}=\Phi_{M/F}\circ\Phi_{E/M},\qquad
 (Q/N)[a]=Q[a]N/N,\qquad Q[a]\cap N=N[\Psi_{M/F}(a)].
\]
The formulas also hold over finite intermediate fields of an
infinite Galois extension.
\end{lemma}

\begin{proof}
Choose the discrete model supplied by
Lemma~\ref{lem:unramified-approximation}. Passing to its inertia
field makes the extension fierce. The quotient and subgroup formulas for
fierce extensions \cite[Remark~3.1(iv)]{Wewers2014}, together with
invariance under unramified base change
\cite[Lemma~3.2]{Wewers2014}, give the formulas for this model:
elements outside inertia have ramification number zero.
The composition identity follows by integrating the orders of
the filtration groups.
The preservation of all intermediate ramification numbers then
gives the formulas for $E/F$. Passing to finite layers gives the
last assertion.
\end{proof}

\subsection{Ramification growth and period matrices}
\label{sec:ramification-periods}
The arguments in this subsection adapt Sen's classical methods for
ramification in $p$-adic Lie extensions and for studying the Lie
algebra of the image of inertia
\cite{Sen1972,Sen1973}; see also
\cite[\S\S0.4.1 and 3.3]{FontaineOuyang}.
The ramification estimates established in the preceding subsection
allow us to carry out these arguments over fields in $\mathcal C$. 

\begin{lemma}\label{lem:type2-inertia-reduction}
Let $K$ be a type~2 field over $\Cp$, and let $L/K$ be Galois with
compact $p$-adic Lie group $D$ and inertia $I$. Put
$U=L^I$, $F=\widehat U$, and $L'=LF\subseteq\widehat L$. Then
\[
 F\in\mathcal C,\qquad \Gal(L'/F)=I,\qquad \widehat{L'}=\widehat L,
\]
and every finite Galois layer of $L'/F$ is fierce.
\end{lemma}

\begin{proof}
The extension $U/K$ is unramified Galois with compact $p$-adic Lie
group $D/I$. Thus $F\in\mathcal C$, and Lemma~\ref{lem:AS-Lie} gives
\[
 [\kappa(F):\kappa(F)^p]=p,\qquad H^1(\kappa(F),\mathbf F_p)\ne0.
\]
The henselian rank-one field $U$ is separably 
closed in $\widehat U$ by Krasner's lemma. Its finite separable extensions
retain their degrees after completion \cite{TemkinNotes}, giving
$\Gal(L'/F)=I$. The inclusions $L\subseteq L'\subseteq\widehat L$ give
$\widehat{L'}=\widehat L$ and $\kappa(L')=\kappa(L)$. Hence $I$ acts
trivially on $\kappa(L')$. Every finite Galois layer $E/F$ has purely
inseparable residue extension and, by
Lemma~\ref{lem:unramified-approximation},
$[E:F]=[\kappa(E):\kappa(F)]$.
\end{proof}

\begin{lemma}\label{lem:field-fierce-basis}
Let $F\in\mathcal C$ and let $E/F$ be finite Galois with purely
inseparable residue extension. Then $E/F$ is fierce and $E\in\mathcal C$.
If $q=[E:F]$ and $z$ is a unit lifting a residue generator, then
\begin{equation}\label{eq:field-orthogonal-basis}
 \OO_E=\bigoplus_{j=0}^{q-1}\OO_Fz^j,\qquad
 \left|\sum_j a_jz^j\right|=\max_j|a_j|.
\end{equation}
For $\sigma\ne1$ in $Q=\Gal(E/F)$,
\begin{equation}\label{eq:field-inertia-number}
 \|\sigma-1\|=|\sigma z-z|,\qquad i_Q(\sigma)=v(\sigma z-z)>0.
\end{equation}
\end{lemma}

\begin{proof}
Lemma~\ref{lem:unramified-approximation} gives the degree equality
and $E\in\mathcal C$. The field $\kappa(F)$ has $p$-rank one, so
$\kappa(E)/\kappa(F)$ has a single generator. Its residue power basis
gives \eqref{eq:field-orthogonal-basis}. Since
$|\sigma(z)^j-z^j|\leq|\sigma z-z|$, expansion in this basis bounds
$\|\sigma-1\|$ by $|\sigma z-z|$; evaluation at $z$ gives equality.
The norm lies in $(0,1)$ because $\sigma$ fixes $\bar z$ but not $z$.
\end{proof}

\begin{lemma}\label{lem:field-abelian-breaks}
Let $A$ be the group of a finite fierce abelian $p$-extension over
$F\in\mathcal C$, and put $A(p^m)=\{g:g^{p^m}=1\}$.
If $b(A)\leq H$, then
\begin{align}
 b(A) &\geq p^{m}b\bigl(A/A(p^m)\bigr),\label{eq:field-small-upper}\\
 u(A)&\geq\frac{p^{m-1}(p-1)}{|A(p)|}
 u\bigl(A/A(p^m)\bigr)\qquad(m\geq1).\label{eq:field-small-lower}
\end{align}
If $b(A)>H$, then
\begin{equation}\label{eq:field-large-upper}
 b(A)=b\bigl(A/A(p)\bigr)+1.
\end{equation}
\end{lemma}

\begin{proof}
Suppose $b=b(A)\leq H$. For $a>b/p$, Proposition~\ref{prop:abelian-power-filtration}
gives $(A[a])^p=1$: use its first formula if $a\leq h$, and its
second if $a>h$. Thus $A[a]\subseteq A(p)$ and quotient compatibility
gives $b(A/A(p))\leq b/p$. Iterate, using
\[
 \bigl(A/A(p)\bigr)(p^{m-1})=A(p^m)/A(p),
\]
to obtain \eqref{eq:field-small-upper}. Moreover,
\[
 u(A)=\int_0^b\frac{da}{|A[a]|}\geq\frac{b-b/p}{|A(p)|},
 \qquad u\bigl(A/A(p^m)\bigr)\leq b/p^m,
\]
which proves \eqref{eq:field-small-lower}.

If $b(A)>H$, Proposition~\ref{prop:abelian-power-filtration} gives, for $a>h$,
\[
 (A[a])^p\ne1\quad\Longleftrightarrow\quad a+1\leq b(A).
\]
Under $A/A(p)\simeq A^p$, the quotient upper group has image
$(A[a])^p$. Its last break is therefore $b(A)-1>h$.
\end{proof}

\begin{lemma}\label{lem:field-last-lower}
Let $F\in\mathcal{C}$, and let $E/F$ be a finite fierce Galois extension whose Galois group $G$ is a nontrivial $p$-group. Set $u=u(G)$. Then the last lower group of $G$ is elementary abelian, and $0<u\leq h$.
\end{lemma}

\begin{proof}
Write $\Delta_\sigma=\sigma-1$ for any $\sigma\in G$. The binomial and commutator identities give
\[
 \|\Delta_{\sigma^p}\|\leq
 \max\{p^{-1}\|\Delta_\sigma\|,\|\Delta_\sigma\|^p\},\qquad
 \|\Delta_{[\sigma,\tau]}\|\leq\|\Delta_\sigma\|\|\Delta_\tau\|.
\]
Together with \eqref{eq:field-inertia-number}, these imply
\begin{align}
 i(\sigma^p)&\geq\min\{i(\sigma)+1,p\,i(\sigma)\}>i(\sigma)
 \quad(\sigma\ne1),\label{eq:field-operator-power}\\
 i([\sigma,\tau])&\geq i(\sigma)+i(\tau).
 \label{eq:field-operator-commutator}
\end{align}
In the last lower group, all nonidentity ramification numbers equal $u>0$;
its $p$-th powers and commutators must therefore be trivial.
If $\sigma$ has order $p^r$, its linear action has an eigenvector $w\ne0$
with eigenvalue a primitive $p^r$-th root $\zeta\in\Cp$. Hence
\[
 p^{-i(\sigma)}\geq\frac{|\sigma w-w|}{|w|}=|\zeta-1|,
 \qquad i(\sigma)\leq\frac1{p^{r-1}(p-1)}\leq h.
\]
This implies that $u\leq h$. Since the extension is fierce, we deduce that $0<u\leq h$.
\end{proof}

Let $L_\infty/F_0$ be a Galois tower with $F_0\in\mathcal C$ such that 
all finite Galois layers are fierce. We further assume that $G:=\Gal(L_\infty/F_0)$ is a $p$-adic Lie group and
\begin{equation}\label{eq:field-uniform-setup}
 G=\exp(\mathfrak l),\qquad [\mathfrak l,\mathfrak l]\subseteq p^3\mathfrak l,
 \qquad d=\operatorname{rank}_{\Zp}\mathfrak l>0,
\end{equation}
where $\mathfrak l$ is a $\Zp$-Lie lattice in $\Lie G$.
Set $G(n)=\exp(p^n\mathfrak l)$, $F_n=L_\infty^{G(n)}$,
$Q_n=G/G(n)$, $b_n=b(Q_n)$, $u_n=u(Q_n)$, and $b_0=u_0=0$.
Each $F_n$ is complete and belongs to $\mathcal C$.

\begin{lemma}\label{lem:field-relative-breaks}
One has $b_m>b_{m-1}$ and $0<u_m\leq h$ for $m\geq1$.
For $m>n$, give $A_{n,m}=G(n)/G(m)$ its numbering relative to $F_n$.
Then
\begin{align}
 u(A_{n,m})&=u_m,\label{eq:field-relative-lower}\\
 b(A_{n,m})&=u_n+b_m-b_n.\label{eq:field-relative-upper}
\end{align}
\end{lemma}

\begin{proof}
Since $(\exp X)^p=\exp(pX)$, the elements of $Q_m$ killed by $p$
are exactly $G(m-1)/G(m)$. Lemma~\ref{lem:field-last-lower} places
the last lower group there. Its image in $Q_{m-1}$ is trivial, so
quotient compatibility at $b_m$ gives $b_{m-1}<b_m$.
The same lemma bounds $u_m$.
The last lower group also lies in $A_{n,m}$; restriction preserves lower
numbers, proving \eqref{eq:field-relative-lower}.
For $\Psi_n=\Phi_{F_n/F_0}^{-1}$, Herbrand composition gives
\[
 b(A_{n,m})=\Psi_n\bigl(\Phi_{F_m/F_0}(u_m)\bigr)
 =\Psi_n(b_m)=u_n+b_m-b_n,
\]
where the last equality uses \eqref{eq:herbrand-tail}.
\end{proof}

\begin{proposition}[Eventual linear growth of upper breaks]\label{prop:linear-break-growth}
Under \eqref{eq:field-uniform-setup}, $b_n=n+c$ for some
$c\in\mathbb R$ and all sufficiently large $n$.
\end{proposition}

\begin{proof}
The bracket condition and the Baker--Campbell--Hausdorff formula give
\begin{equation}\label{eq:field-abelian-window}
 G(n)/G(n+r)\simeq p^n\mathfrak l/p^{n+r}\mathfrak l
 \simeq(\mathbb Z/p^r\mathbb Z)^d\qquad(r\leq n+1);
\end{equation}
see \cite{Lazard1965}.
Fix $r=d+3$. We claim that there are arbitrarily large $n$ for which
$b(G(n)/G(n+r))>H$. Otherwise \eqref{eq:field-small-lower}, with
$m=r-1$, would give
\[
 u_{n+r}\geq p^{r-2-d}(p-1)u_{n+1}=p(p-1)u_{n+1},
\]
since $|A_n(p)|=p^d$ and $A_n/A_n(p^{r-1})=G(n)/G(n+1)$
for $A_n=G(n)/G(n+r)$. Iterating along a progression of step $r-1$
contradicts $0<u_j$ and the bound $u_m\leq h$.

Choose such a sufficiently large $n$. Equation~\eqref{eq:field-large-upper} gives
\[
 b(G(n)/G(n+r))=b(G(n)/G(n+r-1))+1.
\]
By \eqref{eq:field-relative-upper}, this yields, for $N=n+r-1$,
\begin{equation}\label{eq:field-unit-step}
 b_{N+1}=b_N+1.
\end{equation}
Now $B=G(N)/G(N+2)\simeq(\mathbb Z/p^2\mathbb Z)^d$ cannot satisfy
$b(B)\leq H$: otherwise \eqref{eq:field-small-upper} would imply
\[
 b(B/B(p))\leq H/p=h,
 \qquad b(B/B(p))=u_N+b_{N+1}-b_N=u_N+1>h.
\]
The strict inequality uses $h=\frac{1}{p-1}\leq 1$ and $u_N>0$.
Thus \eqref{eq:field-large-upper} applies to $B$ and gives
$b_{N+2}=b_{N+1}+1$. Induction proves the assertion.
\end{proof}

\begin{lemma}\label{lem:field-trace-norm}
Let $E/F$ be a nontrivial finite fierce Galois extension over
$F\in\mathcal C$, of degree $N$. For the minimal polynomial $P$
of a unit residue generator $z$, put $\mathfrak d(E/F)=v(P'(z))$.
Then
\begin{equation}\label{eq:field-different-break}
 \mathfrak d(E/F)=\sum_{\sigma\ne1}i(\sigma)=b(E/F)-u(E/F),
 \qquad \|\Tr_{E/F}\|=p^{-\mathfrak d(E/F)}.
\end{equation}
If $N=p^q$ and $T=p^{-q}\Tr_{E/F}$, then
\begin{equation}\label{eq:field-normalized-trace}
 \|T\|=p^{q-\mathfrak d(E/F)},\qquad
 v(Tw)\geq v(w)-q+\mathfrak d(E/F)\quad(w\in E).
\end{equation}
\end{lemma}

\begin{proof}
The identity $P'(z)=\prod_{\sigma\ne1}(z-\sigma z)$ gives the sum,
and $\Phi(u)=u+\sum_{\sigma\ne1}i(\sigma)$ gives $b=\mathfrak d+u$.
The trace-dual of $\OO_E=\OO_F[z]$ is $P'(z)^{-1}\OO_E$.
Choose $a\in F^\times$ with $v(a)=\mathfrak d$, using
$v(E^\times)=v(F^\times)$. For $x\in\OO_E$, the element $x/a$
belongs to the trace-dual, so $v(\Tr(x))\geq\mathfrak d$.
Conversely,
\[
 \Tr_{E/F}\left(\frac{z^{N-1}}{P'(z)}\right)=1
\]
exhibits the unit $az^{N-1}/P'(z)$ with trace $a$.
This proves the norm equality and its normalized form.
\end{proof}

\begin{lemma}\label{lem:field-ax-sen}
Fix $C_{\mathrm{Ax}}>p/(p-1)^2$. For a complete rank-one valued field $F$ such that $\operatorname{char}(F)=0$, if $x\in\widehat{\overline F}$ satisfies
$v(\sigma x-x)\geq m$ for all $\sigma\in G_F$, then there exists some $y\in F$
suth that  $v(x-y)\geq m-C_{\mathrm{Ax}}$.
If $L/F$ is algebraic Galois and $N\subseteq\Gal(L/F)$ is closed, then
$(\widehat L)^N=\widehat{L^N}$.
\end{lemma}

\begin{proof}
Apply the Ax--Sen approximation and invariant theorems
\cite[Propositions~3.3 and~3.8]{FontaineOuyang}.
The approximation statement extends from algebraic elements to their
completion by choosing an algebraic $x'$ with $v(x-x')>m$.
For the invariant statement, use the inverse image of $N$ in $G_F$.
\end{proof}

\begin{lemma}\label{lem:bounded-coboundary-conductor}
Let $L_\infty/F$ be a Galois tower over $F\in\mathcal C$ with
Galois group $V$ such that all finite Galois layers are fierce. Suppose
$\lambda:V\to\Qp$ is a continuous map and $x\in\widehat{L_\infty}$ satisfies
\begin{equation}\label{eq:field-approximate-coboundary}
 v\bigl(\lambda(\sigma)-(\sigma-1)x\bigr)\geq m
 \qquad(\sigma\in V)
\end{equation}
for an integer $m$. Then $\chi=p^{-m}\lambda\bmod\Zp$ is a finite-image
continuous character. If $\chi\ne0$ and $E=L_\infty^{\ker\chi}$, then
\begin{equation}\label{eq:field-bounded-conductor}
 \mathfrak d(E/F)\leq C_{\mathrm{Ax}},\qquad
 c(\chi)=b(E/F)\leq C_{\mathrm{Ax}}+h.
\end{equation}
\end{lemma}

\begin{proof}
For $e_\sigma=\lambda(\sigma)-(\sigma-1)x$,
\[
 \lambda(\sigma\tau)-\lambda(\sigma)-\lambda(\tau)
 =e_{\sigma\tau}-e_\sigma-\sigma e_\tau\in p^m\Zp.
\]
The right side has valuation at least $m$ and the left side lies in $\Qp$.
Thus $\chi$ is a homomorphism. Since $V$ is compact, the image of $\chi$ is finite cyclic.
If $\chi\ne0$, 
set $s=\min_{\sigma\in V}v(\lambda(\sigma))$. Then $s<m$ and its image has order $p^q$, where $q=m-s$.

Put $J=\ker\chi$. For $\tau\in J$, one has $v((\tau-1)x)\geq m$.
Since the image of $G_E$ on $L_\infty$ is $J$,
Lemma~\ref{lem:field-ax-sen} gives
\[
 y\in E,\qquad v(x-y)\geq m-C_{\mathrm{Ax}},\qquad
 v\bigl(\lambda(\sigma)-(\sigma-1)y\bigr)\geq m-C_{\mathrm{Ax}}.
\]
Here $E/F$ is finite and hence complete. Its normalized trace $T$
annihilates $(\sigma-1)y$, since $J\triangleleft V$, and fixes
$\lambda(\sigma)\in F$. Equation~\eqref{eq:field-normalized-trace}
therefore gives
\[
 v(\lambda(\sigma))\geq m-C_{\mathrm{Ax}}-q+\mathfrak d(E/F).
\]
For an element attaining $s$, this becomes
\[
 s\geq m-C_{\mathrm{Ax}}-(m-s)+\mathfrak d(E/F)
   =s-C_{\mathrm{Ax}}+\mathfrak d(E/F).
\]
Thus $\mathfrak d(E/F)\leq C_{\mathrm{Ax}}$.
Equation~\eqref{eq:field-different-break} and
Lemma~\ref{lem:field-last-lower} give $b(E/F)\leq C_{\mathrm{Ax}}+h$.
The induced character of $\Gal(E/F)$ is faithful, so $c(\chi)=b(E/F)$.
\end{proof}

\begin{lemma}\label{lem:period-inertia}
Let $K$ be a type~2 field over $\Cp$, and let $L/K$ be Galois with
$p$-adic Lie Galois group $D$ and inertia $I$.
If a continuous representation $\rho:D\to\GL_r(\Qp)$ admits
\begin{equation}\label{eq:field-period-matrix}
 M\in\GL_r(\widehat L),\qquad
 \sigma(M)=M\rho(\sigma)\quad(\sigma\in D),
\end{equation}
then $\rho(I)$ is finite. In particular, $d\rho(\Lie I)=0$.
\end{lemma}

\begin{proof}
Set $U=L^I$, $F=\widehat U$, and $L'=LF$.
Lemma~\ref{lem:type2-inertia-reduction} gives $F\in\mathcal C$,
$\Gal(L'/F)=I$, $\widehat{L'}=\widehat L$, and fierce finite layers.
For $N=\ker(\rho|_I)$, Lemma~\ref{lem:field-ax-sen} gives
$M\in\GL_r(\widehat{(L')^N})$.
Replace the tower by $(L')^N/F$, whose group is the compact matrix
group $\rho(I)$. Suppose $\dim\rho(I)>0$.
Choose a sufficiently small open normal uniform subgroup
$G=\exp(\mathfrak l)$ with $[\mathfrak l,\mathfrak l]\subseteq p^3\mathfrak l$.
Its fixed field $F_0$ is finite Galois over $F$ and belongs to $\mathcal C$.
Denote this tower by $L_\infty/F_0$; its finite layers remain fierce.
The representation is now the matrix inclusion, with $\sigma(M)=M\sigma$.
Use $G(n),F_n,b_n,u_n$ from \eqref{eq:field-uniform-setup}.

For $\sigma\in G(n)$, one has $\sigma-1=O(p^n)$ and
$\sigma(M)-M=O(p^n)$, with constants uniform in $n,\sigma$.
Since $G_{F_n}$ has image $G(n)$ on $L_\infty$, entrywise Ax--Sen gives
\[
 M_n\in M_r(F_n),\qquad M_n-M=O(p^n).
\]
For large $n$, the inverses $M_n^{-1}$ are uniformly bounded. Hence
\[
 C_n=M_n^{-1}M=1+O(p^n),\qquad \sigma(C_n)=C_n\sigma,
\]
and
\begin{equation}\label{eq:field-linearized-period}
 \sigma(C_n)-C_n
 =(\sigma-1)+(C_n-1)(\sigma-1)=\log\sigma+O(p^{2n}).
\end{equation}
Set $x_{ij,n}=(C_n)_{ij}$ and
$\lambda_{ij,n}(\sigma)=(\log\sigma)_{ij}$.
There is a common integer $C\geq0$ such that
\[
 v\bigl(\lambda_{ij,n}(\sigma)-(\sigma-1)x_{ij,n}\bigr)
 \geq m_n,\qquad m_n=2n-C,
\]
for all coordinates and $\sigma\in G(n)$, once $n$ is large.
Lemma~\ref{lem:bounded-coboundary-conductor}, over $F_n$, gives
\begin{equation}\label{eq:field-coordinate-conductors}
 \chi_{ij,n}=p^{-m_n}\lambda_{ij,n}\bmod\Zp,\qquad
 c(\chi_{ij,n})\leq C_{\mathrm{Ax}}+h.
\end{equation}
These conductors are relative to $F_n$; set $c(0)=0$.

By Proposition~\ref{prop:linear-break-growth}, $b_n=n+c$ for large $n$.
Fix $B>C_{\mathrm{Ax}}+h$ and put $a_n=n+c+B+1$ and $R_n=G[a_n]$.
Since $a_n>b_n$, one has $R_n\subseteq G(n)$. Herbrand compatibility
and \eqref{eq:herbrand-tail} give
\[
 R_n=G(n)[\Psi_n(a_n)],\qquad
 \Psi_n(a_n)=u_n+a_n-b_n=u_n+B+1>B.
\]
Thus all $\chi_{ij,n}$ vanish on $R_n$, and
\begin{equation}\label{eq:field-deep-matrix-log}
 \log R_n\subseteq p^{m_n}M_r(\Zp).
\end{equation}
The two lattices $\mathfrak l$ and
$(\Qp\mathfrak l)\cap M_r(\Zp)$ are commensurable. Choose a fixed
integer $b\geq0$ with
$(\Qp\mathfrak l)\cap M_r(\Zp)\subseteq p^{-b}\mathfrak l$.
For $C'=C+b$, equation~\eqref{eq:field-deep-matrix-log} yields
$R_n\subseteq G(2n-C')$. Therefore the upper group of $G/G(2n-C')$
is trivial at $a_n$, and
\[
 b_{2n-C'}<a_n.
\]
But Proposition~\ref{prop:linear-break-growth} makes this
$2n-C'+c<n+c+B+1$, a contradiction for large $n$.
Thus $\dim\rho(I)=0$; compactness implies that $\rho(I)$ is finite.
\end{proof}

\begin{corollary}\label{cor:perfectoid-period-obstruction}
Let $F$ be a type~2 field over $\Cp$, and let $L/F$ be Galois with
compact $p$-adic Lie group $D$ and inertia $I$.
If a continuous $\rho:D\to\GL_r(\Qp)$ has injective differential and a matrix as in
\eqref{eq:field-period-matrix}, then $I$ is finite and $\widehat L$
is not perfectoid.
\end{corollary}

\begin{proof}
Lemma~\ref{lem:period-inertia} gives $\Lie I=0$, so $I$ is finite.
Put $U=L^I$. The residue field $\kappa(F)$ has $p$-rank one;
the separable algebraic extension $\kappa(U)/\kappa(F)$ preserves a
$p$-basis. Since $L/U$ is finite, $\kappa(L)/\kappa(U)$ is finite and
also preserves $p$-rank. Thus $[\kappa(L):\kappa(L)^p]=p$.
Completion leaves this residue field unchanged, whereas a perfectoid
field has perfect residue field.
\end{proof}

\subsection{Surjectivity on the curve}\label{sec:curve-conclusion}

\begin{proof}[Proof of Theorem~\ref{thm:curve}]
Suppose that the geometric Sen map is not surjective at a type~2
point. Proposition~\ref{prop:type2-reduction} gives a type~2 field
$K$, a Galois extension $L/K$ with compact $p$-adic Lie group $D$,
and a continuous representation $\rho:D\to\GL_r(\Qp)$ with
injective differential, together with
\[
 M\in\GL_r(\widehat L),\qquad
 \sigma(M)=M\rho(\sigma)\quad(\sigma\in D).
\]
It also gives that $\widehat L$ is perfectoid, contradicting
Corollary~\ref{cor:perfectoid-period-obstruction}.
Thus the Sen map is surjective
at every type~2 point.

By Lemma~\ref{lem:analytic-rank-locus}, its nonsurjectivity locus
is closed analytic. Every nonempty open subset of a smooth
$\Cp$-curve contains a type~2 point, so this locus contains no
irreducible component. A proper closed analytic subset of a smooth
curve is locally finite and consists of classical points. Its
complement is therefore Zariski open and dense, and the subset is
finite if $X$ is quasicompact.
\end{proof}

\section{Locally analytic decompletion and Sen surjectivity}
In this section, we study the relation between Sen surjectivity
and locally analytic decompletion.
We first recall the announced comparison results that yield
the implication from Sen surjectivity to locally analytic
decompletion. We then prove the converse implication
independently of these comparisons.
\subsection{The analytic Hodge--Tate stack and decompletion}
\label{sec:forward}

The analytic Hodge--Tate stack $X^{\mathrm{HT}}$ is the Hodge--Tate
part of the analytic prismatization of Ansch\"utz, Le Bras,
Rodr\'\i guez Camargo, and Scholze
\cite{AnschutzLeBrasRodriguezCamargoScholzePrismatization}.

We use the following comparison from the joint work
\cite{AnschutzLeBrasRodriguezCamargoScholzePrismatization}, announced in
\cite[Section~1.1]{RodriguezCamargoCartierDuality}.

\begin{theorem}\label{thm:ht-perfect-comparison}
There is a natural symmetric monoidal equivalence
\begin{equation}\label{eq:ht-perfect-comparison}
 \mathrm D_{\mathrm{perf}}(X^{\mathrm{HT}})
 \simeq \mathrm D_{\mathrm{perf}}(X_v,\Oh_X),
\end{equation}
compatible with pullback and restricting to
\[
 \Vect(X^{\mathrm{HT}})\simeq\Vect(X_v,\Oh_X).
\]
\end{theorem}

\begin{remark}
On a toric chart, the vector-bundle comparison is realized by relative
Sen decompletion: one descends from the completed toric tower to a
finite-level Sen module and differentiates its $\Gamma$-action.
The resulting commuting operators give the corresponding Higgs field.
See
\cite[Theorem~2.4.4 and Proposition~3.2.3]{RodriguezCamargo2026}.
\end{remark}

For a perfectoid profinite \'etale $G$-torsor $\widetilde X\to X$, write
$G^{\mathrm{an}}$ and $\widetilde X^{\mathrm{an}}$ for the locally analytic
group stack and the locally analytic model of the tower in the following
comparison. We use Rodr\'\i guez Camargo's announced characterization
\cite[approximately 1:15:00]{RodriguezCamargoGeometricSenLecture}.

\begin{theorem}\label{thm:ht-quotient}
There is a canonical morphism of analytic stacks
\begin{equation}\label{eq:ht-comparison}
 \alpha_\pi:X^{\mathrm{HT}}\longrightarrow
       [\widetilde X^{\mathrm{an}}/G^{\mathrm{an}}].
\end{equation}
It is an isomorphism if and only if $\theta_{\widetilde X}$ is
surjective.
\end{theorem}

\begin{corollary}\label{prop:forward}
If $\theta_{\widetilde X}$ is surjective, then pullback gives an
equivalence
\begin{equation}\label{eq:vb-decompletion}
 \Vect([\widetilde X^{\mathrm{an}}/G^{\mathrm{an}}])
 \simeq \Vect(X_v,\Oh_X).
\end{equation}
In the affinoid realization of this comparison, let $X$ be small,
let $\widetilde X$ be affinoid perfectoid, and put
$B=\Oh_X(\widetilde X)$. For every finite projective continuous
semilinear $B$-representation $W$ of $G$, the module $W^{\la}$ is
finite projective over $B^{\la}$ and multiplication is an isomorphism
\begin{equation}\label{eq:decompletion}
 B\otimes_{B^{\la}}W^{\la}\xrightarrow{\ \sim\ }W.
\end{equation}
\end{corollary}

\begin{proof}
Combine $\alpha_\pi^*$ with the vector-bundle restriction of
Theorem~\ref{thm:ht-perfect-comparison}.
Descent along $\widetilde X\to X$ identifies the right side of
\eqref{eq:vb-decompletion} with $G$-equivariant vector bundles on
$\widetilde X$, and the quotient on the left identifies its bundles with
$G^{\mathrm{an}}$-equivariant bundles on $\widetilde X^{\mathrm{an}}$.
In the affinoid vector-bundle realization of the announced comparison,
the locally analytic model associated with $W$ has module of sections
$W^{\la}$, finite projective over $B^{\la}$, and pullback to the completed
tower is extension of scalars along $B^{\la}\to B$.
These affinoid identifications are included in the comparison input.
Its pullback identification is precisely the multiplication map
\eqref{eq:decompletion}.
\end{proof}

\begin{remark}
We recall the description of $X^{\mathrm{HT}}$ for smooth $X$ from
\cite[Definition~3.4.1 and Remark~3.4.2]{RodriguezCamargoCartierDuality}.
Put
\[
 T_X=(\Omega^1_{X/\Cp})^\vee,\qquad
 T_X^\dagger(1)=
 T_X^{\mathrm{an}}(1)\otimes_{\mathbb G_a^{\mathrm{an}}}
                         \mathbb G_a^\dagger,
\]
where $\mathbb G_a^\dagger$ is the overconvergent neighbourhood of zero.
For a partially proper base, the comparison
\[
 R\Gamma_v(X,\Oh_X)
   \simeq R\Gamma_!(X,\mathbb G_a^{\mathrm{an,dR}})
\]
transports the Hodge--Tate class to $!$-cohomology. The boundary for
\[
 \mathbb G_a^\dagger\longrightarrow\mathbb G_a^{\mathrm{an}}
     \longrightarrow\mathbb G_a^{\mathrm{an,dR}},
\]
gives $\eta_{\mathrm{HT}}\in H^2_!(X,T_X^\dagger(1))$.
For general $X$, this construction is performed on Huber's
compactification and pulled back to $X$, as in the cited
\cite[Remark~3.4.2]{RodriguezCamargoCartierDuality}. On an affinoid $X$, the initial class corresponds to
$\mathrm{id}_{\Omega^1_X}$ under
\[
 H^1_v\bigl(X,T_X(1)\otimes_{\OO_X}\Oh_X\bigr)
   \simeq H^0\bigl(X,T_X\otimes_{\OO_X}\Omega^1_{X/\Cp}\bigr).
\]
The stack is defined by the Cartesian square
\begin{equation}\label{eq:ht-construction}
\begin{CD}
 X^{\mathrm{HT}} @>>> X\\
 @VVV @VV{e}V\\
 X @>{\eta_{\mathrm{HT}}}>> B^2T_X^\dagger(1),
\end{CD}
\end{equation}
where $e$ is the zero section; thus $X^{\mathrm{HT}}\to X$ is a
$BT_X^\dagger(1)$-torsor.
\end{remark}

\subsection{Decompletion implies surjectivity}\label{sec:converse}

\begin{theorem}
\label{thm:decompletion-surjectivity}
Let $X=\Spa(A,A^\circ)$ be a small smooth affinoid of dimension $d$
over $\Cp$, let $G$ be a compact $p$-adic Lie group, and let
$\widetilde X\to X$ be a diamondian affinoid perfectoid profinite
\'etale $G$-torsor. Put $B=\Oh_X(\widetilde X)$. Suppose that, for
every finite free $B$-module $W$ with continuous semilinear
$G$-action, the multiplication map
\[
 B\otimes_{B^{\la}}W^{\la}\longrightarrow W
\]
is surjective. Then $\theta_{\widetilde X}$ is surjective on
$X_{\proet}$.
\end{theorem}

We prepare the comparison of sections needed in the proof.
Let $X=\Spa(A,A^\circ)$ be small and let $\widetilde X\to X$ be a
diamondian affinoid perfectoid profinite \'etale $G$-torsor.
Let $P$ be the affinoid perfectoid space representing
$\widetilde X^\diamond$, and put $B=\Oh_X(\widetilde X)$.
For a fixed toric chart on $X$, use the tower $X_\infty$ and the
notation $A_{\mathrm{tor}}$, $\Gamma$, $\eta_j$, and $\omega_j$
of Subsection~\ref{sec:sen-operators}.

To compare the torsor map with the toric operators, put
\[
 \widetilde X_\infty=\widetilde X\times_X X_\infty,\qquad
 S=\Oh_X(\widetilde X_\infty).
\]
The $G$ and $\Gamma$ actions on $S$ commute.
The object $\widetilde X_\infty$ is affinoid perfectoid in
$X_{\proet}$: it is a profinite \'etale refinement of the affinoid
perfectoid toric tower, so this follows from
\cite[Lemma~4.5\textup{(i)} and the proof of Lemma~4.6]{Scholze2013}.
Let $Z$ be its associated affinoid perfectoid space. Then
\[
 Z^\diamond\simeq
 P^\diamond\times_{X^\diamond}X_\infty^\diamond,
\]
and $Z\to P$ is a $\Gamma$-torsor for the $v$-topology.

\begin{lemma}\label{lem:invariants}
There are canonical topological isomorphisms
\begin{equation}\label{eq:invariants}
 B\simeq\OO(P)\simeq S^\Gamma.
\end{equation}
For every finite projective $B$-module $W$, the natural map
$W\to(S\otimes_BW)^\Gamma$ is a topological isomorphism.
Moreover,
\begin{equation}\label{eq:diamondian-cech-vanishing}
 H^1_{\mathrm{cont}}(\Gamma,S)=0.
\end{equation}
\end{lemma}

\begin{proof}
The \v Cech nerve of the pro-\'etale cover
$\widetilde X_\infty\to\widetilde X$ has terms
$\widetilde X_\infty\times\Gamma^n$ for $n\geq0$.
Each term is affinoid perfectoid and has completed sections
$C^0(\Gamma^n,S)$ by \cite[Lemma~4.10\textup{(iii)}]{Scholze2013}.
Its associated diamond is the corresponding term
$Z\times\Gamma^n$ of the \v Cech nerve of $Z\to P$.
Taking the sheaf equalizer on $X_{\proet}$ gives $B=S^\Gamma$,
while $v$-descent for the structure sheaf gives
$\OO(P)=S^\Gamma$
\cite[Theorem~17.1.3]{ScholzeWeinstein2020}.
The same equalizers for $\Oh_X^+$ and $\OO^+$ identify
$\Oh_X^+(\widetilde X)$ with $\OO^+(P)$, so the identifications
respect the natural topologies. In particular, $B$ is a perfectoid
Banach $\Cp$-algebra. Expressing a finite projective $W$ as an
idempotent summand of $B^r$ then gives its asserted invariant
identification, including the topology.

Every term $Z\times\Gamma^n$ is affinoid perfectoid and hence
acyclic for $\OO$ on the $v$-site. Thus the continuous cochain
complex $C^0(\Gamma^\bullet,S)$ computes $R\Gamma_v(P,\OO)$.
Since $P$ is affinoid perfectoid, the same acyclicity theorem gives
\[
 H^1_{\mathrm{cont}}(\Gamma,S)=H^1_v(P,\OO)=0.
\]
\end{proof}

Choose a basis $\xi_1,\ldots,\xi_r$ of $\Lie G$. After pullback to
$\widetilde X_\infty$, write
\begin{equation}\label{eq:sen-matrix}
 \theta_{\widetilde X,S}(\xi_\ell^\vee)
   =\sum_j a_{j\ell}\omega_j,\qquad a_{j\ell}\in S.
\end{equation}
The two projections of the $(G\times\Gamma)$-torsor
$\widetilde X_\infty\to X$ and \eqref{eq:sen-functoriality}
identify its Sen map with
\[
 (\lambda,\mu)\longmapsto
 \theta_{\widetilde X,S}(\lambda)+\theta_{X_\infty,S}(\mu).
\]
Its dual sends the vector dual to $\omega_j$ to
$\sum_\ell a_{j\ell}\xi_\ell+\eta_j$.
The Sen annihilation theorem
\cite[Corollary~3.2.6]{RodriguezCamargo2026} therefore gives
\begin{equation}\label{eq:sen-annihilation}
 \eta_j(s)+\sum_\ell a_{j\ell}\xi_\ell(s)=0
 \qquad\bigl(s\in S^{\jla}\bigr).
\end{equation}

\begin{proof}[Proof of Theorem~\ref{thm:decompletion-surjectivity}]
Consider the representation of $\Gamma=\Zp^d$
\[
 V=\Qp e_0\oplus\bigoplus_{i=1}^d\Qp e_i,
 \qquad \gamma(e_0)=e_0,
 \qquad \gamma(e_i)=e_i+\gamma_i e_0.
\]
Let $\mathbb V$ be its local system via $X_\infty\to X$, with the
associated-local-system convention giving the diagonal action on a
trivializing cover. Set $\mathcal E=\Oh_X\otimes_{\Qp}\mathbb V$.
There is an exact sequence
\[
 0\longrightarrow\Oh_Xe_0\longrightarrow\mathcal E
 \xrightarrow{q}\Oh_X^{\,d}\longrightarrow0.
\]
Sheaf descent along $\widetilde X_\infty\to\widetilde X$ identifies
\[
 W:=\mathcal E(\widetilde X)
   =(S\otimes_{\Qp}V)^\Gamma,
\]
where $\Gamma$ acts diagonally. For each $i$, the map
$\gamma\mapsto-\gamma_i$ is a continuous $S$-valued cocycle.
Lemma~\ref{lem:invariants} gives $h_i\in S$ satisfying
\[
 \gamma(h_i)-h_i=-\gamma_i\qquad(\gamma\in\Gamma).
\]
The vectors $e_0$ and $f_i=e_i+h_ie_0$ are $\Gamma$-invariant
and form an $S$-basis of $S\otimes_{\Qp}V$. Since $S^\Gamma=B$,
they form a $B$-basis of $W$. Consequently,
\begin{equation}\label{eq:test-sequence}
 0\longrightarrow Be_0\longrightarrow W\xrightarrow{q}B^d
 \longrightarrow0.
\end{equation}
The vectors $f_i$ split this sequence, so $W$ is finite free of
rank $d+1$. The action of $G$ on coefficients commutes with
$\Gamma$ and fixes the $e_i$. For $g\in G$,
\[
 g(f_i)=f_i+\bigl(g(h_i)-h_i\bigr)e_0,
 \qquad g(h_i)-h_i\in S^\Gamma=B.
\]
Since $B$ has the subspace topology from $S$, these formulas give
$W$ a continuous semilinear $G$-action. The same basis gives
\begin{equation}\label{eq:test-trivialization}
 S\otimes_BW\simeq S\otimes_{\Qp}V.
\end{equation}
Here $G$ acts on coefficients and fixes the $e_i$, while $\Gamma$
acts on both coefficients and $V$. Sections of $W$ are
$\Gamma$-invariant in \eqref{eq:test-trivialization}.

For $w\in W^{\la}$ write
\[
 w=b_0e_0+\sum_i b_ie_i\quad\text{in }S\otimes_{\Qp}V.
\]
The continuous $G$-equivariant coefficient maps show that every $b_i$
is $G$-locally analytic. Invariance under $\Gamma$ gives
\begin{equation}\label{eq:unipotent-coefficients}
 \gamma(b_i)=b_i\ (i\ge1),\qquad
 \gamma(b_0)=b_0-\sum_i\gamma_i b_i.
\end{equation}
Thus $b_i\in S^\Gamma=B$ for $i\ge1$ and
$q(w)=(b_1,\ldots,b_d)$. The explicit orbit formula
\[
 (g,\gamma)(b_0)=g(b_0)-\sum_i\gamma_i g(b_i)
\]
proves that $b_0\in S^{\jla}$. Differentiating
\eqref{eq:unipotent-coefficients} gives
\begin{equation}\label{eq:coefficient-derivative}
 \eta_j(b_0)=-b_j.
\end{equation}
Apply \eqref{eq:sen-annihilation} to $b_0$ and use
\eqref{eq:coefficient-derivative}. For each $j$ this gives
\[
 0=\sum_\ell a_{j\ell}\xi_\ell(b_0)+\eta_j(b_0)
   =\sum_\ell a_{j\ell}\xi_\ell(b_0)-b_j.
\]
Equivalently,
\begin{equation}\label{eq:analytic-image}
 \sum_j b_j\omega_j
 =\theta_{\widetilde X,S}
   \left(\sum_\ell\xi_\ell(b_0)\xi_\ell^\vee\right).
\end{equation}
Thus the image of $q(W^{\la})$, viewed in $S^d$ using the basis
$\omega_j$, lies in $\im\theta_{\widetilde X,S}$.

By hypothesis, $W^{\la}$ generates the module $W$ in
\eqref{eq:test-sequence} over $B$. Since $q$ is surjective,
$q(W^{\la})$ generates $B^d$, hence generates $S^d$ after scalar
extension. Equation~\eqref{eq:analytic-image} implies that
$\theta_{\widetilde X,S}$ is surjective. The map
$\widetilde X_\infty\to X$ is a pro-\'etale cover, so surjectivity
of \eqref{eq:sen-map} follows by the locality of surjectivity for the
pro-\'etale topology.
\end{proof}

\begin{remark}\label{rem:counterexample-decompletion}
For the counterexample constructed in \S \ref{subsec: construction counter ex},
Theorem~\ref{thm:decompletion-surjectivity} shows that locally
analytic decompletion fails near the origin. More precisely,
after restricting to a small affinoid neighborhood of the origin,
there exists a finite free continuous semilinear representation
that is not generated by its locally analytic vectors.

\end{remark}

\begin{corollary}\label{thm:equivalence}
Let $X=\Spa(A,A^\circ)$ be a small smooth affinoid over $\Cp$, and let
$G$ be a compact $p$-adic Lie group. Let $\widetilde X\to X$ be a
profinite \'etale $G$-torsor which is affinoid perfectoid as an object
of $X_{\proet}$ in the sense of
\cite[Definition~4.3\textup{(i)}]{Scholze2013}, and put
$B=\Oh_X(\widetilde X)$. The following conditions are equivalent.
\begin{enumerate}
\renewcommand{\labelenumi}{(\arabic{enumi})}
\item The map $\theta_{\widetilde X}$ is surjective on $X_{\proet}$.
\item For every finite projective $B$-module $W$ with continuous
semilinear $G$-action, $W^{\la}$ is finite projective over $B^{\la}$ and
the multiplication map
\[
 B\otimes_{B^{\la}}W^{\la}\xrightarrow{\ \sim\ }W
\]
is an isomorphism.
\item For every finite free $B$-module $W$ with continuous semilinear
$G$-action, the multiplication map
$B\otimes_{B^{\la}}W^{\la}\to W$ is surjective.
\end{enumerate}
\end{corollary}

\begin{proof}
Corollary~\ref{prop:forward} gives \textup{(1)} implies \textup{(2)},
and \textup{(2)} implies \textup{(3)} immediately.
Theorem~\ref{thm:decompletion-surjectivity} gives \textup{(3)}
implies \textup{(1)}.
\end{proof}

\begin{remark}
For affinoid perfectoid torsors, the converse should have a
Tannakian interpretation: an appropriate reconstruction theorem
for analytic stacks would recover $\alpha_\pi$ from its symmetric
monoidal pullback on vector bundles. It is presently unclear to
us how to apply such a reconstruction theorem to this setting.
Theorem~\ref{thm:decompletion-surjectivity} uses a single unipotent
test representation.
\end{remark}

\begingroup
\raggedright
\raggedbottom
\bibliographystyle{amsalpha}
\bibliography{geometric_sen}
\endgroup
\end{document}